\documentclass[11pt]{amsart}
\usepackage{amssymb,latexsym,graphicx, amscd, enumitem}

\newtheorem{theorem}{Theorem}[section]
\newtheorem{prop}[theorem]{Proposition}
\newtheorem{lemma}[theorem]{Lemma}
\newtheorem{remark}[theorem]{Remark}
\newtheorem{question}[theorem]{Question}
\newtheorem{definition}[theorem]{Definition}
\newtheorem{cor}[theorem]{Corollary}

\DeclareMathOperator{\ep}{\epsilon}

\DeclareMathOperator{\C}{\mathbb C}

\begin{document}

\title{Intersection of almost complex submanifolds}

\author{Weiyi Zhang}
\address{Mathematics Institute\\  University of Warwick\\ Coventry, CV4 7AL, England}
\email{weiyi.zhang@warwick.ac.uk}

\begin{abstract} We show the intersection of a compact almost complex subvariety of dimension $4$ and a compact almost complex submanifold of codimension $2$ is a $J$-holomorphic curve. This is a generalization of positivity of intersections for $J$-holomorphic curves in almost complex $4$-manifolds to higher dimensions. As an application, we discuss pseudoholomorphic sections of a complex line bundle. We introduce a method to produce $J$-holomorphic curves using the differential geometry of almost Hermitian manifolds. When our main result is applied to  pseudoholomorphic maps, we prove the singularity subset of a pseudoholomorphic map between almost complex $4$-manifolds is $J$-holomorphic. Building on this, we show degree one pseudoholomorphic maps between almost complex $4$-manifolds are actually birational morphisms in pseudoholomorphic category.  
\end{abstract}
 \maketitle

\tableofcontents
\section{Introduction}
A pseudoholomorphic curve (to the author's knowledge, first studied in \cite{NW})  is a smooth map from a Riemann surface into an almost complex manifold $(M, J)$ that satisfies the Cauchy-Riemann equation. Gromov \cite{Gr} first introduced this notion as a fundamental tool to study symplectic manifolds. It has since revolutionized the field of symplectic topology and greatly influenced many other areas such as algebraic geometry, string theory, and $4$-manifolds. The image of a $J$-holomorphic curve is called a $J$-holomorphic $1$-subvariety (or misleadingly also called a $J$-holomorphic curve), whose complete definition will be recalled shortly. It is the analogue of a one dimensional subvariety in algebraic geometry. 

The positivity of intersection of distinct irreducible $J$-holomorphic curves is a fundamental result in the theory of $J$-holomorphic curves \cite{Gr, McD, MW}, in particular when the ambient manifold is of dimension $4$. On the other hand, the intersection theory of complex submanifolds, or more generally complex subvarieties, is a well established subject. In particular, it is a basic result that the intersection of complex submanifolds is a possibly singular complex subvariety. This is clear from the ``mapping out" viewpoint: namely, a complex submanifold could be expressed locally as the zero locus of analytic functions in terms of complex coordinates.  This leads to the  divisor-line bundle correspondence in complex geometry, where we view the sections as complex codimension one submanifolds in the total space of the line bundle. 

Most research in the theory of $J$-holomorphic curves,  like the Gromov-Witten theory, take the ``mapping into" viewpoint. On the other hand, Taubes' SW=Gr uses the ``mapping out" viewpoint with a limit process, where $J$-holomorphic curves are constructed as limits of the zero loci of solutions of Seiberg-Witten equations. The goal of our paper is to develop the ``mapping out" approach and see how it might also help to produce deeper understanding of the ``mapping into" viewpoint. 

As we have seen in the complex setting, the ``mapping out" approach is essentially the intersection theory of almost complex submanifolds. Hence, our study could also be extended to higher dimensional $J$-holomorphic subvarieties.  However, the almost complex case is much harder since we no longer have complex coordinates. 

Before we start to introduce our results, we first define $J$-holomorphic subvarieties. A $J$-holomorphic subvariety is a finite set of pairs $\{(V_i, m_i), 1\le i\le m\}$, where each $V_i$ is an irreducible $J$-holomorphic subvariety and each $m_i$ is a positive integer. Here an irreducible $J$-holomorphic subvariety is the image of a somewhere immersed pseudoholomorphic map $\phi: X\rightarrow M$ from a compact connected smooth almost complex manifold $X$. 

 We have an equivalent description of irreducible subvarieties of dimension $2$, {\it i.e.} irreducible $1$-subvarieties, as in \cite{T1}. A closed set $C\subset M$ with finite, nonzero
$2$-dimensional Hausdorff measure is said to be an irreducible $J$-holomorphic $1$-subvariety if it has no isolated points, and if the complement of a finite set
of points in $C$, called the singular points, is a connected smooth submanifold with $J$-invariant tangent space. Any irreducible $1$-subvariety is the image of a $J$-holomorphic map $\phi: \Sigma\rightarrow M$ from a compact connected curve $\Sigma$ where  $\phi$ is an embedding off a finite set. These two definitions are equivalent because any $J$-holomorphic curve $u: \Sigma\rightarrow M$ may be expressed as a composition of a holomorphic branched covering $b: \Sigma\rightarrow \Sigma'$ and a somewhere injective $J$-holomorphic map $u': \Sigma'\rightarrow M$.

The simplest situation of the intersection of almost complex submanifolds is when there is no ``excess intersection" phenomenon, see the discussion following Corollary \ref{eICdim4}. For this reason, we require one of the submanifolds to be of codimension $2$. 

\begin{question}\label{intJcurve}
Suppose $(M^{2n}, J)$ is an almost complex $2n$-dimensional manifold, and $Z_2$ is a compact connected almost complex submanifold of codimension $2$. If the intersection $Z_1\cap Z_2$ is not one of $Z_i$, is it a $J$-holomorphic subvariety of dimension $\dim_{\mathbb R} Z_1-2$?
\end{question}

The statement is apparently true if $Z_1$ and $Z_2$ intersect transversely, or the intersection is known to be a smooth manifold. It is well known that if a connected compact $J$-holomorphic curve is not contained in a connected compact codimension $2$ almost complex submanifold then their intersection is a finite set, see Lemma \ref{posint}. This is the simplest form of positivity of intersections, a phenomenon first noticed by Gromov in \cite{Gr}. In dimension $4$, the strongest form is known. Any intersection point of two distinct irreducible $J$-holomorphic subvarieties contributes positively \cite{MW, McD}.

As the next step, we are able to give an affirmative answer to Question \ref{intJcurve} when $\dim Z_1=4$. In fact, we have the following more general form. 

\begin{theorem}\label{ICdim4}
Suppose $(M^{2n}, J)$ is an almost complex $2n$-manifold, and  $Z_2$ is a codimension $2$ compact connected almost complex submanifold. Let $(M_1, J_1)$ be a compact connected almost complex $4$-manifold and $u: M_1\rightarrow M$ a pseudoholomorphic map such that $u(M_1)\nsubseteq Z_2$. Then $u^{-1}(Z_2)$ supports a $J_1$-holomorphic $1$-subvariety in $M_1$.
\end{theorem}

Notice we do not require our almost complex structures $J$ or $J_1$ to be tamed by a symplectic form. Recall that an almost complex structure $J$ is said to be
tamed by a symplectic form $\omega$ if the bilinear form $\omega(\cdot, J\cdot)$ is positive definite. We say $J$ is tamed if we do not specify such a symplectic form albeit there exists one. An almost complex structure $J$ is compatible with $\omega$ if $J$ is tamed by $\omega$ and $\omega(v, w)=\omega(Jv, Jw)$ for any $v, w\in TM$. We also say $J$ is almost K\"ahler if we do not specify such a symplectic form. 

In the statement of Theorem \ref{ICdim4}, a set supporting a pseudoholomorphic $1$-subvariety means it is the support $|\Theta|=\cup_{(C_i, m_i)\in \Theta}C_i$ of a pseudoholomorphic $1$-subvariety $\Theta$. In fact, we are also able to determine the homology class of the $J_1$-holomorphic $1$-subvariety in Theorem \ref{ICdim4}.  The homology class $e_{\Theta}=\sum_{(C_i, m_i)\in \Theta} m_i[C_i]$ is calculated by the homology class of the submanifold of $M_1$ that is obtained in a similar manner but using a (smooth) perturbation of $u$ that is transverse to $Z_2$. 

Notice that the image of $u$ might not be an irreducible $J$-holomorphic subvariety of dimension $4$. If $u(M_1)$ is of dimension $0$, then it is a point, and $u(M_1)\cap Z_2=\emptyset=u^{-1}(Z_2)$ since $u(M_1)\nsubseteq Z_2$. If $u(M_1)$ is a $J$-holomorphic $1$-subvariety, since $u(M_1)\nsubseteq Z_2$ and $u(M_1)$ is connected, then $u(M_1)\cap Z_2$ is a collection of finitely many points possibly with multiplicities, {\it i.e.} a $0$-dimensional subvariety. If $u(M_1)$ is of dimension $4$, $u(M_1)\nsubseteq Z_2$ is the image of $J_1$-holomorphic subvarieties $u^{-1}(Z_2)$. While each irreducible component of $u^{-1}(Z_2)$ is either contracted to a point, or mapped to a $J$-holomorphic curve. 

We remark that if $n=2$, then the statement of Theorem \ref{ICdim4} still holds even if $Z_2$ is merely assumed to be a $J$-holomorphic $1$-subvariety, by virtue of the result of \cite{MW}. This is the content of Section \ref{Z2subv} and is summarized as Theorem \ref{ICdim4MW}. This generality is useful to study pseudoholomorphic maps between almost complex $4$-manifolds.

A quick corollary of Theorem \ref{ICdim4} which suffices for many applications is the following.

\begin{cor}\label{eICdim4}
Suppose $(M^{2n}, J)$ is an almost complex $2n$-dimensional manifold, and $Z_1$, $Z_2$ are compact connected almost complex submanifolds of dimension $4$ and $2n-2$ respectively. Then the intersection $Z_1\cap Z_2$ is either one of $Z_i$, or supports a $J$-holomorphic $1$-subvariety. 
\end{cor}

In general, $Z_1$ and $Z_2$ might not intersect transversely. But, if $Z_1\cap Z_2\ne Z_1$ or $Z_2$,  they are ``dimensional transverse" in the sense that $\dim_{\mathbb R} Z_1+\dim_{\mathbb R} Z_2=\dim_{\mathbb R} M+\dim_{\mathbb R}Z_1\cap Z_2$. Our codimension $2$ assumption on $Z_2$ is used to guarantee that there are no ``excess intersection" besides the trivial case. For example, projective subspaces of complex dimensions $k$ and $l$ in $\mathbb CP^n$ could share common projective subspaces of any dimension no less than $k+l-n$. Moreover, the intersection could have irreducible components with unequal dimensions.

Let us briefly explain the idea of the proof of Theorem \ref{ICdim4}. The first step is to show $u^{-1}(Z_2)$ has finite $2$-dimensional Hausdorff measure. To establish this, we first introduce a generalization of unique continuation of pseudoholomorphic curves. This prevents $u^{-1}(Z_2)$ from being an open subset of $M_1$, thus reducing our argument to a small open neighborhood of any point in $M_1$. Then we use a dimension reduction argument. Notice a dimension $2$ and a codimension $2$ almost complex submanifolds intersect at isolated points positively. This could be used to show the $2$-dimensional Hausdorff measure of the intersection $A=u^{-1}(Z_2)$ is finite with the help of a local smooth foliation of $J_1$-holomorphic disks on $M_1$. Then the coarea formula would imply the finiteness of $2$-dimensional Hausdorff measure of $A$. This part will be done in Section \ref{H2fin}.

If in addition, roughly speaking, we know the set $A$ intersects positively with all local $J_1$-holomorphic disks, then we can show $A$ is a $J_1$-holomorphic subvariety. The basic strategy dated back to \cite{King} at least, where it works in complex analytic setting. In the pseudoholomorphic situation, this strategy was worked out by Taubes \cite{T}. He introduces the notion of ``positive cohomology assignment", which plays the role of intersection number of our set $A$ with each local open disk. If, in addition to $\mathcal H^2(A)<\infty$, $A$ has a ``positive cohomology assignment", we know $A$ is a $J_1$-holomorphic subvariety by Proposition 6.1 of \cite{T} (see Proposition \ref{pcaholo} in our paper). Then the argument is boiled down to finding a positive cohomology assignment when $A$ is considered as a subset in $4$-manifold $M_1$. The idea is, instead of using the set $A$ directly, we assign the intersection number of the image of our test disk in $M$ with the submanifold $Z_2$ in the ambient manifold $M$. This part occupies Section \ref{pca}. In Section \ref{homology}, we calculate the homology class of the $J_1$-holomorphic subvariety $A$. 

When $\dim_{\mathbb R}M_1>4$, up to the higher dimensional analogue of Proposition \ref{pcaholo} (Question \ref{pcaholoco2}), our argument still works to show the $J_1$-holomorphicity of $A$. This occupies Section \ref{dim>4}, and the result is summarized in Theorem \ref{ICdim>4}. 

In the remaining sections, we discuss the applications of Theorem \ref{ICdim4}. In Section \ref{Jsec}, we study the pseudoholomorphic sections of a complex line bundle. We study the almost complex canonical bundle $\Lambda_J^-$ in detail. This is defined as the $-1$ eigenspace of the endomorphism of $\Lambda^2T^*M$ induced by the almost complex structure $J$.  In this case, any almost Hermitian metric of the base manifold $(M, J)$ induces an almost complex structure of the total space of the complex line bundle $\Lambda_J^-$. The pseudoholomorphic sections of this bundle correspond to $J$-anti-invariant forms with certain closedness condition. Moreover,  we partially establish the divisor-line bundle correspondence in this section. The upshot of this section is to relate almost Hermitian geometry to the theory of pseudoholomorphic curves. In principle, combining with Theorem \ref{ICdim4}, the study of the differential geometry of almost Hermitian manifolds would lead to the information of the $J$-holomorphic subvarieties on the base manifold, and {\it vice versa}. 

In Section \ref{app}, we discuss more applications. In addition to those related to symplectic birational geometry which are summarized in Section \ref{symbir}, we study the pseudoholomorphic maps between almost complex manifolds. The following result studies the structure of pseudoholomorphic maps between closed almost complex $4$-manifolds.

\begin{theorem}\label{introsing4to4}
Let $u: (X, J)\rightarrow (M, J_M)$ be a somewhere immersed pseudoholomorphic map between closed connected almost complex $4$-manifolds. Then 
\begin{itemize}
\item the singularity subset of $u$ supports a $J$-holomorphic $1$-subvariety; 
\item other than finitely many points $x\in M$, where $u^{-1}(x)$ is the union of a $J$-holomorphic $1$-subvariety and finitely many points, the preimage of each point is a set of finitely many points.
\end{itemize}
\end{theorem}

The singularity subset of $u$ is where the differential $du_p$ is not of full rank. 
It is a combination of Theorem \ref{s4to4} and Proposition \ref{4to4fiber}. Every point of $X$ is a singularity if $u$ is nowhere immersed. 

A closer study at degree one pseudoholomorphic maps between almost complex $4$-manifolds shows that they are eventually birational morphisms in pseudoholomorphic category. First, Zariski's main theorem still holds for pseudoholomorphic maps (Proposition \ref{Zarmain}). Moreover, we have a very concrete description of the exceptional set. It is summarized in the following, which is a combination of Theorem \ref{phblowdown} and Corollary \ref{diffbd}.

\begin{theorem}\label{introphblowdown}
Let $u: (X, J) \rightarrow (M, J_M)$ be a degree one pseudoholomorphic map between closed connected almost complex $4$-manifolds such that $J$ is almost K\"ahler. Then there exists a subset $M_1\subset M$, consisting of finitely many points, with the following significance:
\begin{enumerate}
\item The restriction $u|_{X\setminus u^{-1}(M_1)}$ is a diffeomorphism. 
\item At each point of $M_1$, the preimage is an exceptional curve of the first kind.  
\item $X\cong M\#k\overline{\mathbb CP^2}$ diffeomorphically, where $k$ is the number of irreducible components of the $J$-holomorphic $1$-subvariety $u^{-1}(M_1)$.
\end{enumerate}
\end{theorem}

Roughly, a connected $J$-holomorphic $1$-subvariety is called an exceptional curve of the first kind if its configuration is equivalent to the empty set through topological blowdowns. See Definition \ref{exc1st}. In particular, it is a connected $J$-holomorphic $1$-subvariety whose irreducible components are rational curves and the dual graph is a tree.

To show the second part, we first establish Grauert's criterion for exceptional set, {\it i.e.} the intersection matrix of the irreducible components of the exceptional set is negative definite. This is the reason that the almost K\"ahler condition  is added in the statement. With this assumption, we are able to embed a neighborhood of the exceptional set into a rational surface, which in particular has $b^+=1$. However, we believe this assumption should be removable.

\subsection{Philosophy and some further directions}
The philosophy of the paper is 

\bigskip

{\it a statement for smooth maps between smooth manifolds in terms of  R. Thom's transversality should also have its counterpart in pseudoholomorphic setting without requiring the transversality or genericity, but using the notion of pseudoholomorphic subvarieties,}

\bigskip

\noindent in particular when such a statement is available in complex analytic setting. Our paper explores a few, among many more, such directions guided by this philosophy.

For instance, the corresponding statement of Corollary \ref{eICdim4} in smooth category is Thom's transversality theorem: If two submanifolds intersect transversely, then the intersection is a smooth manifold. Its complex counterpart is the intersection theory of analytic cycles. In all the three ({\it i.e.} smooth, holomorphic, and pseudoholomorphic) categories, this is the cornerstone of all the later discussions. 

When it is applied to sections of an oriented vector bundle over an oriented manifold, we know the zero locus of a transverse section is a submanifold of the base whose homology class is Poincar\'e dual to the Euler class of the vector bundle. For a holomorphic line bundle, the zero divisor of a holomorphic section is a divisor in the first Chern class of the line bundle. This motivates the discussion in Section \ref{Jsec}. 

Our philosophy also leads to the following variant of pseudoholomorphic sections of the canonical line bundle. It is known that for a generic Riemannian metric on a $4$-manifold, a self-dual harmonic $2$-form, {\it i.e.} a section of the bundle $\Lambda_g^+$ which is closed as a $2$-form, is symplectic off a disjoint union of embedded circles, with the latter being the vanishing locus of the form. The almost complex version of this is the following question in \cite{DLZiccm}. 

\begin{question}\label{Janti0holo}
For an almost complex structure $J$ on a $4$-manifold, is a $J$-anti-invariant closed $2$-form ({\it i.e.} a section of the complex line bundle $\Lambda_J^-$ which is closed as a $2$-form) almost K\"ahler off a $J$-holomorphic $1$-subvariety? Equivalently, is the zero locus of a $J$-anti-invariant closed $2$-form a $J$-holomorphic $1$-subvariety?
\end{question}

 We remark that for any almost Hermitian metric $g$, the bundle $\Lambda_J^-$ is a rank $2$ subbundle of the rank $3$ real vector bundle $\Lambda_g^+$. Apparently, the answer is yes when $J$ is integrable. This question is answered affirmatively for an arbitrary almost complex structure in \cite{BonZ}.

Our Theorem \ref{introsing4to4} has two parts. The first part is a finer version of Sard's theorem, which says the set of critical values has measure zero. The second part finds its counterpart in Thom's and Boardman's fundamental work on singularities of differentiable maps \cite{Th, Boa}. It says for generic smooth maps between smooth manifolds, the singularity subsets with given degeneracy data are  submanifolds of the domain. The complex analytic version of Theorem \ref{introphblowdown} stated for equidimensional map is the fact that a birational morphism is a composition of a sequence of blowing down along exceptional curves.

Most of our statements are for $1$-subvarieties. The main reason is that the properties of pseudoholomorphic $1$-subvarieties are well studied. For higher dimensional subvarieties, even the definition is not systematically explored. Interestingly, some techniques used in this paper are also powerful in higher dimensions. We will explore it in a sequel, albeit one might already find some of such analysis in Section \ref{app}. It is an interesting iteration that would reward us with generalizing our results to higher dimensions once higher dimensional pseudoholomorphic  subvarieties are studied. 

Another interesting and peculiar phenomenon missing in the transversality setting is the excess intersection of (pseudo)holomorphic submanifolds as we have mentioned. It would have many interesting applications, for example, to the structures of pseudoholomorphic maps $u: (X^{2k}, J_X)\rightarrow (M^{2n}, J_M)$ with $k<n$ by applying the same strategy of Theorem \ref{s4to4}.

Since Thom's transversality is such a powerful tool which is applied to vast areas, we would expect many other pseudoholomorphic counterparts guided by our philosophy. A notable direction is the pseudoholomorphic version of the Thom-Mather topological stability theorem. 

\subsection{Acknowledgements}
This work is partially supported by EPSRC grant EP/N002601/1. We would like to thank Tedi Draghici, Cheuk Yu Mak and Mario Micallef for helpful discussions, Louis Bonthrone for careful reading and the referee for many nice suggestions to improve the presentation of this paper.

\section{Finite $2$-dimensional Hausdorff measure in dimension $4$}\label{H2fin}
We assume $u: (M_1, J_1)\rightarrow (M, J)$ is a $(J_1, J)$-holomorphic map and $A=u^{-1}(Z_2)$. We denote the image $Z_1:=u(M_1)$. It needs not be embedded. In this section, we will show that the Hausdorff $\dim_HA=2$ and the Hausdorff measure $\mathcal H^2(A)$ is finite when $\dim M_1=4$. This is the first step towards the proof of Theorem \ref{ICdim4}. We begin with a couple of preparatory lemmas, which work for any dimension.

\begin{lemma}\label{compact}
Let $u: M_1\rightarrow M$ be a continuous map with $M_1$ compact, $Z_2$ a closed subset of a (not necessarily compact) manifold $M$. Then $A=u^{-1}(Z_2)\subset M_1$ is a compact set.
\end{lemma}
\begin{proof}
Since $Z_2$ is a closed subset in $M$, we know the preimage $A$ is closed in $M_1$. Since a closed subset of a compact space is compact, we know $A$ is a compact set. 
\end{proof}

The next lemma is the easiest case of positivity of intersections, a phenomenon first noticed by Gromov in \cite{Gr}. Our statement is adapted from Exercise 2.6.1 of \cite{MS}. For a proof, see \cite{Wen}.

\begin{lemma}\label{posint}
Suppose that $Q$ is a compact codimension two $J$-holomorphic submanifold of the almost complex manifold $(M, J)$, and let $u: D\rightarrow (M, J)$ be a $J$-holomorphic curve such that $u(0)\in Q$. 
\begin{enumerate}
\item Then the inverse image of the intersection points $u^{-1}(Q)$ is isolated in $D$ except in the case when $u(D)\subset Q$. 
\item Suppose $u(D)\nsubseteq Q$. Shrinking $D$, if necessary, we may assume that $u(0)$ is the unique point where $u(D)$ meets $Q$. Define the local intersection number $u\cdot Q$ of $u$ with $Q$ to be the number of points of intersection of a generic smooth perturbation of $u$ relative to the boundary $\partial D$. Then $u\cdot Q\ge 1$ with equality if and only if $u$ is transverse to $Q$ at zero. 
\end{enumerate}
\end{lemma}

The first part of the lemma could be extended to a generalization of unique continuation of $J$-holomorphic submanifolds.

\begin{prop}\label{uniquecont}
Let $Y$ be a compact $J$-holomorphic submanifold of the almost complex manifold $(M, J)$, $u: (X, J_1)\rightarrow (M, J)$ a $(J_1, J)$-holomorphic map with $X$ connected and compact. If  $Y$ contains the image of an open subset of $X$, then $u(X)\subset Y$.
\end{prop}
\begin{proof}
If $X$ is a $J$-holomorphic curve, it is the standard unique continuation result of pseudoholomorphic curves which could be obtained by Carleman similarity principle, Aronszajn's result or Hartman-Wintner theorem \cite{MS, Wen}. 

When $\dim_{\mathbb R}X\ge 4$, the proof could be reduced to the pseudoholomorphic curve situation. We assume $A\subset X$ is the maximal open subset contained in $u^{-1}(Y)$. The points of $A$ could be characterized as follows: a point $x\in A$ if and only if there is an open neighborhood $\mathcal N(x)$ of it in $X$ such that $\mathcal N(x)$ is also contained in $u^{-1}(Y)$. 

If $A\ne X$, since $u^{-1}(Y)$ is compact, the complement of $u^{-1}(Y)$ in $X$ is also an open subset of $X$. We call it $B$. By definition, $A\cup \partial A\cup B=X$. Moreover, the set $\partial A=u^{-1}(Y)\setminus A$ is nowhere dense. We choose a point $x\in \partial A\cap \partial B$. There is such an $x$ if $B\ne \emptyset$. 
 
 Now, by Lemma 6.1 in \cite{Ush} (when $\dim X=4$, it follows from Lemma 6.6 of \cite{T}), there exists an open neighborhood $\mathcal N_x'\subset X$ of $x$, such that we have a smooth map $f_0: D\times U\rightarrow \mathbb \mathcal N_x$ with $f_0(0, U)=x$. Here $D\subset \mathbb C$ is a unit disk and $U\subset \mathbb C^{n-1}$ is a unit ball of radius $1$. 
For each  $\kappa \in U$, the map $f_0|_{D \times \kappa}$ is an embedding whose image is a $J_1$-holomorphic disk. Moreover,  $f_0|_{(D\setminus \{0\})\times U}$ is a diffeomorphism onto its image. We can thus choose some $\kappa$ such that $f_0|_{D \times \kappa}\cap A\ne \emptyset$. 
 
By Lemma \ref{diskfol} (when $\dim X=4$, it follows from Lemma 5.4 of \cite{T}), there exists an open neighborhood $\mathcal N_x\subset X$ of $x$, such that we have a smooth map $f: D\times U\rightarrow \mathbb \mathcal N_x$ with $f(0, 0)=x$. For each  $w\in U$, the map $f|_{D \times w}$ is an embedding whose image is a $J_1$-holomorphic disk. We can choose $f$ such that $f(D\times 0)=f_0(D\times \kappa)$.  Moreover,  $f$ is a diffeomorphism onto its image.

For any $w\in U$, when $B\neq \emptyset$, we claim the disk $f(D \times w)$ cannot contain nontrivial open subsets from two of the three disjoint sets: $A$, $B$ and $\partial A$. First, if $f(D \times w)$ intersects both $A$ and $B$, since $A$ and $B$ are both open, the intersections $f(D \times w)\cap B$ and $f(D \times w)\cap A$ are open subsets of $f(D \times w)$. Since $(f(D \times w)\cap B)\cap u^{-1}(Y)=\emptyset$, it implies $u(f(D \times w))\nsubseteq Y$. However, on the other hand, $f(D \times w)$ is a $J_1$-holomorphic curve which has an open subset $f(D \times w)\cap A$ whose image under $u$ is contained in $Y$. Hence, by unique continuation of $J$-holomorphic curves, $u(f(D \times w))\subset Y$. This contradiction implies $B=\emptyset$ and hence $A=X$.

Furthermore, if the disk $f(D \times w)$ contains a nontrivial open subset $\partial A\cap f(D \times w)$, we will have $w'$ and $w''$ arbitrarily close to $w$ such that $f(D \times w')$ and $f(D \times w'')$ intersect $A$ and $B$ respectively. Hence,  $f(D \times w)$ cannot contain any point from $A$ or $B$. This completes our claim in last paragraph. We will call the disk $f(D \times w)$ with $f(D \times w)\cap A\ne \emptyset$ an $A$-type disk. 

Now, since $f(D\times 0)$ is an $A$-type disk, there is an $\ep>0$, such that when $|w|<\ep$, we know all the disks $f(D \times w)$ are of $A$-type. Hence, the image $f({D\times \{w: |w|<\ep\}})$ is an open subset of $X$ contained in  $u^{-1}(Y)$. It implies $x$ is not an accumulation point of $B$ which contradicts to our choice of $x$. 

Hence $B=\emptyset$ and thus $u(X)\subset Y$. 
\end{proof}

Now we assume $\dim M_1=4$. We would study the intersections of $A=u^{-1}(Z_2)$ with $J_1$-holomorphic disks by Lemma \ref{posint}. The plan requires a $2$-dimensional family of such disks. These disks could be constructed from perturbations of $J_0$-holomorphic ones where $J_0$ is the standard complex structure of $\mathbb C^2$. Such a construction is worked out in section 5(d) of \cite{T}.

To begin, fix a point $x\in M_1$, we find an open neighborhood of $x$ such that $J_1$ is compatible with a non-degenerate $2$-form $\Omega$ in this neighborhood. The pair $(\Omega, J_1)$ induces an almost Hermitian metric. Then we choose Gaussian normal coordinates centered at $x$ which identify a geodesic ball about $x$ with a ball in $\mathbb R^4$ and take $x$ to the origin. We identify $\mathbb R^4=\mathbb C^2$ such that $\Omega|_x=\omega_0=dx^1\wedge dx^2+dx^3\wedge dx^4=\frac{i}{2}(dw_0\wedge d\bar{w}_0+dw_1\wedge d\bar{w}_1)$. Here we use $(w_0, w_1)$ for the complex coordinates on $\mathbb C^2$. Hence we will simply say the almost complex structure $J_1$ is on $\mathbb C^2$. It agrees with the standard almost complex structure $J_0$ at the origin, but typically nowhere else.

Denote a family of holomorphic disks $D_w:=\{(\xi, w)| |\xi|<\rho\}$, where $w\in D$. What we get from \cite{T}, mainly Lemma 5.4 (see Lemma \ref{diskfol} for a higher dimensional generalization using the same argument), is a diffeomorphism $f: D\times D \rightarrow \mathbb C^2$, where $D\subset \mathbb C$ is the disk of radius $\rho$, such that: 
\begin{itemize}
 \item For all $w\in D$, $f(D_w)$ is a $J_1$-holomorphic submanifold containing $(0, w)$. 
 \item For all $w\in D$, dist$((\xi, w); f(\xi, w))\le z\cdot \rho\cdot |\xi|$. Here $z$ depends only on $\Omega$ and $J_1$.
\item For all $w\in D$, the derivatives of order $m$ of $f$ are bounded by $z_m\cdot \rho$, where $z_m$ depends only on $\Omega$ and $J_1$.
\end{itemize}

We call such a diffeomorphism $J$-fiber-diffeomorphism. We have freedom to choose the ``direction" of these disks by rotating the Gaussian coordinate system. Each $J_1$-holomorphic disk can be chosen to be close to any complex affine planes foliation of $\mathbb C^2$ with direction $(a, b)$, at least near the complex line (passing through the origin) with direction $(b, -a)$.

\begin{prop}\label{Hd2}
Suppose $(M^{2n}, J)$ is an almost complex $2n$-dimensional manifold, and  $Z_2$ is a codimension $2$ compact connected almost complex submanifold. Let $(M_1, J_1)$ be a compact connected almost complex $4$-manifold and $u: M_1\rightarrow M$ a pseudoholomorphic map such that $u(M_1)\nsubseteq Z_2$. Then $A=u^{-1}(Z_2)$ has Hausdorff dimension $\dim_HA=2$ and the Hausdorff measure of $A$, $\mathcal H^2(A)$, is finite. 
\end{prop}
\begin{proof}
Since $M_1$ is compact, the Hausdorff measure will be independent of the choice of Hermitian metric. In fact, we will measure the set $A$ locally by the metric induced from the local coordinate we choose above. For any point $x\in A\subset M_1$, we can find a diffeomorphism $f^x$ onto an open subset of $M_1$ as above. The union of the images $f^x(D^x\times D^x)$ covers the set $A$. By Lemma \ref{compact}, $A$ is compact. Hence we can choose only finitely many $x_i$ such that these $f^{x_i}(D^{x_i}\times D^{x_i})$ covers the set $A$. We could assume all these $D^{x_i}$ have radius $\rho$. To show the $2$-dimensional Hausdorff measure $\mathcal H^2(A)<\infty$, we only need to show it for $A\cap f^{x_i}(D\times D)$ of each $x_i$. 

Now we will show $\mathcal H^2(A\cap f^{x}(D\times D))<\infty$ for any $x\in A$. We will simply write $f$ instead of $f^x$. 
Since for each $w\in D$, $f(D_w)$ is a $J_1$-holomorphic disk in $M_1$, we know it intersects $u^{-1}(Z_2)$ at finitely many points if it is not totally contained in $u^{-1}(Z_2)$ by Lemma \ref{posint}. Furthermore, we claim that there are only finitely many $w\in \bar D$ such that $f(D_w)$ is contained in $u^{-1}(Z_2)$. 

If it is not the case, we could assume without loss of generality that $0$ is an accumulation point of these $w$. Now we construct $J$-fiber-diffeomorphism in another direction. The goal is to foliate a neighborhood of $x$ by $J$-holomorphic disks transverse to $f(D_0)$. As above, we take coordinates centered at $x$ such that $x$ is the origin and $(0, w')$ is identified with $f(D_0)$. Then we choose a $J$-fiber-diffeomorphism $f': D'\times D'\rightarrow \mathbb C^2$, where $D'\subset \mathbb C$ is the disk of radius $\rho'<\rho$, such that 

\begin{itemize}
 \item For all $w'\in D'$, $f'(D'_{w'})$ is a $J_1$-holomorphic submanifold containing $(0, w')$. 
\item For all $w'\in D'$, dist$((\xi', w'); f'(\xi', w'))\le z\cdot \rho'\cdot |\xi'|$. Here $z$ depends only on $\Omega$ and $J_1$.
\item For all $w'\in D'$, the derivatives of order $m$ of $f'$ are bounded by $z_m\cdot \rho'$, where $z_m$ depends only on $\Omega$ and $J_1$.
\end{itemize}

In particular, all the disks $f'(D'_{w'})$ are transverse to $f(D_0)=f'(0 \times D')=0\times D'$. As being transverse is an open condition, $f'(D'_{w'})$ are transverse to $f(D_w)$ for all $|w|\le \epsilon$. Hence the intersection points of $f'(D'_{w'})$ and $u^{-1}(Z_2)$ are not isolated. By Lemma \ref{posint}, the whole disks $f'(D'_{w'})\subset u^{-1}(Z_2)$. In turn, we have $f'(D'\times D')\subset u^{-1}(Z_2)$. Since $f'(D'\times D')$ covers an open neighborhood of $x$ in $M_1$, we know  $u(M_1)\subset Z_2$ by Proposition \ref{uniquecont}. This contradicts to the assumption of our proposition. Hence, we have established our claim that  there are only finitely many $w\in \bar D$ such that $f(D_w)$ is contained in $u^{-1}(Z_2)$.

Moreover, we can actually choose our diffeomorphism $f$ such that none of the $J_1$-holomorphic disks $f(D_w)$ is contained in $u^{-1}(Z_2)$. We first show that, for any point $x\in M_1$, there are only finitely many complex directions of $T_xM_1$ such that there are $J_1$-holomorphic curves tangent to it and  contained in $u^{-1}(Z_2)$. Suppose there are infinitely many. Since the complex directions of $T_xM_1$ are parametrized by $\mathbb CP^1$, we know there is at least one direction, $v$, which is accumulative. By the perturbative nature of $J$-fiber-diffeomorphism, we can choose the local Gaussian coordinates such that $f(D_0)$ is transverse to $v$. Hence, for $|w|<\epsilon$, $f(D_w)$ are transverse to $v$ as well. In particular, the intersection numbers of  $u(f(D_w))$ and $Z_2$ are infinite which contradicts  Lemma \ref{posint} and Proposition \ref{uniquecont}. To summarize, we have proven that there are only finitely many complex directions of $T_xM_1$ such that there are $J_1$-holomorphic curves in $u^{-1}(Z_2)$ tangent to it. 

Hence, fixing $x$, we can choose a complex direction such that there is no $J$-holomorphic curve in $u^{-1}(Z_2)$ tangent to it. By the perturbative nature of $J$-fiber-diffeomorphisms, we can choose the local Gaussian coordinates and diffeomorphism $f$ such that no $f(D_w)$ is contained in $Z_2$ when $|w|$ is sufficiently small.

Now we estimate the Hausdorff measure of the set $A\cap  f(\bar D\times \bar D)$. The intersection $A\cap f(\bar D\times \bar D)$ is also compact. 
Furthermore, since $f$ is a diffeomorphism, we can choose $D$ smaller if necessary such that the distortion of $f$ at the larger open domain $2D\times 2D$ is bounded by a positive constant $C$. By our choice of the local coordinates and the diffeomorphism $f$, $A\cap f(\bar D_w)$ is a set of finitely many points for each $w\in \bar D$.
Look at the function $g: \bar D\rightarrow \mathbb N\cup \{0\}$ from the base disk $\bar D$ to non-negative integers whose value $g(w)$ is the intersection number of $u\circ f(\bar D_w)$ and $Z_2$. This is an upper semi-continuous function. Hence, it achieves maximal value at some point $w\in \bar D$, say $N$. 
Since each intersection point contributes positively by Lemma \ref{posint}, we know $A\cap f(\bar D_w)$ contains at most $N$ points for all $w\in \bar D$. Since  $A\cap  f(\bar D\times \bar D)$ is compact, we cover it by finitely many balls of radius $\epsilon$. By Vitali covering lemma, we can choose a subset of these balls which are disjoint to each others, say there are $L$ such balls, such that the union of the $L$ concentric balls with radius $3\epsilon$ covers the set $A\cap  f(\bar D\times \bar D)$. Each $\epsilon$-ball intersects $f( 2D_w)$ at an open set of area no greater than $\pi C^2\epsilon^2$. By coarea formula, we have $$N\cdot \pi C^2\epsilon^2\cdot \pi C^2(2\rho)^2>L \frac{1}{2}\pi^2\epsilon^4.$$ In other words, there are no more than $C'\cdot \epsilon^{-2}$ many balls with radius $3\epsilon$ covering $A\cap  f(\bar D\times \bar D)$. Thus the $2$-dimensional Hausdorff measure $\mathcal H^2(A\cap  f(\bar D\times \bar D))<\infty$, which in turn implies $\mathcal H^2(A)<\infty$.
\end{proof}

\section{$J$-holomorphic intersection subvariety}
In this section, we will first finish the proof of Theorem \ref{ICdim4} and then calculate the homology class of the intersection subvariety. We will also discuss two generalizations of Theorem \ref{ICdim4}. One is Theorem \ref{ICdim4MW}, which is a combination of \cite{MW} and our Theorem \ref{ICdim4}.  The other is the higher dimensional version, Theorem \ref{ICdim>4}.
\subsection{Positive cohomology assignment}\label{pca}
In this subsection, we will prove Theorem \ref{ICdim4} using the notion of positive cohomology assignment, which is introduced in \cite{T}. We assume $(X, J)$ is an almost complex manifold, and $C\subset X$ is merely a subset at this moment. Let $D\subset \mathbb C$ be the standard unit disk. A map $\sigma: D\rightarrow X$ is called {\it admissible} if $C$ intersects the closure of $\sigma(D)$ inside $\sigma(D)$. Next we recall the notion of a positive cohomology assignment to $C$, which is extracted from section 6.1(a) of \cite{T}.

\begin{definition}\label{PCA}
A positive cohomology assignment to the set $C$ is an assignment of an integer, $I(\sigma)$, to each admissible map $\sigma: D\rightarrow X$. Furthermore, the following criteria have to be met: 
\begin{enumerate}
\item If $\sigma: D\rightarrow X\setminus C$, then $I(\sigma)=0$. 

\item If $\sigma_0, \sigma_1: D\rightarrow X$ are admissible and homotopic via an admissible homotopy (a homotopy $h:[0, 1]\times D\rightarrow X$ where $C$ intersects the closure of Image$(h)$ inside Image$(h)$), then $I(\sigma_0)=I(\sigma_1)$.

\item Let $\sigma: D\rightarrow X$ be admissible and let $\theta: D\rightarrow D$ be a proper, degree $k$ map. Then $I(\sigma\circ \theta)=k\cdot I(\sigma)$.

\item Suppose that $\sigma: D\rightarrow X$ is admissible and that $\sigma^{-1}(C)$ is contained in a disjoint union $\cup_iD_i\subset D$ where each $D_i=\theta_i(D)$ with $\theta_i: D_i\rightarrow D$ being an orientation preserving embedding. Then $I(\sigma)=\sum_iI(\sigma\circ \theta_i)$.

\item If $\sigma: D\rightarrow X$ is admissible and a $J$-holomorphic embedding with $\sigma^{-1}(C)\ne \emptyset$, then $I(\sigma)>0$.
\end{enumerate}
\end{definition}

The following is Proposition 6.1 of \cite{T}.

\begin{prop}\label{pcaholo}
Let $(X, J)$ be a $4$-dimensional almost complex manifold and let $C\subset X$ be a closed set with finite $2$-dimensional Hausdorff measure and a positive cohomology assignment. Then $C$ supports a compact $J$-holomorphic $1$-subvariety. 
\end{prop}

Taubes' proof of Proposition \ref{pcaholo} could be understood as consisting of the following two steps. First, he proves that, under the assumptions, the set $C$ gives an almost complex integral $2$-cycle  (see \cite{RT}):
\medskip

\begin{enumerate} 
\item {\it Rectifiability}: There exists an at most countable union of disjoint oriented $C^1$ $2$-submanifolds $\mathcal C=\cup_iN_i$ and an integer multiplicity $\theta\in L_{\mbox{loc}}^1(\mathcal C)$ such that for any smooth compactly supported $2$-form $\psi$ one has $$C(\psi)=\sum_i\int_{N_i}\theta\psi.$$
\item {\it Closedness}: $\partial C=0$, {\it i.e.} $\forall \alpha\in \mathcal D^{1}(M), C(d\alpha)=0$.
\item {\it Almost complex}: For $\mathcal H^2$ almost every point $x\in \mathcal C$, the approximate tangent plane $T_x$ to the rectifiable set $\mathcal C$ is invariant under the almost complex structure $J$, {\it i.e.} $J(T_x)=T_x$.
\end{enumerate}

\medskip

In fact, this step is eventually Lemma 6.10 of \cite{T}, which shows that an open dense subset of $C$ has the structure of a Lipschitz submanifold of $X$.

The second step is to show that any integral $2$-dimensional almost complex cycle $C$ could be realized by a $J$-holomorphic subvariety $\Theta=\{(C_i, m_i)\}$ in the sense that $C(\psi)=\sum_i m_i\int_{C_i}\psi$. The latter result is generalized in \cite{RT} to any $2p$-dimensional almost complex manifold $(M, J)$ satisfying the locally symplectic property. It could also be derived from Almgren's big regularity paper \cite{Alm} and S. Chang's PhD thesis \cite{Ch}. 
Recall we say $(X, J)$ has the {\it locally symplectic} property if in a neighborhood of each point $x\in X$, there exists a symplectic form compatible with $J$.
 It was shown in \cite{Lej, RTcre}\footnote{This result was first noticed in \cite{RTcre} Lemma A.1. However, the proof is incomplete since it relies on a wrong claim of Peter J. Olver. A complete proof is given in \cite{Lej}.}  that any $4$-dimensional almost complex manifold $(X, J)$ has the locally symplectic property. However, a general higher dimensional almost complex manifold is not locally symplectic. 

Let us return to the setting of Theorem \ref{ICdim4}. Suppose $(M^{2n}, J)$ is an almost complex $2n$-dimensional manifold, and  $Z_2$ is a codimension $2$ compact connected almost complex submanifold. Let $M_1$ be a compact connected almost complex $4$-manifold and $u: M_1\rightarrow M$ a pseudoholomorphic map such that $u(M_1)\nsubseteq Z_2$. 
In Proposition \ref{Hd2}, we have shown that $A= u^{-1}(Z_2)$ is a closed set with finite $2$-dimensional Hausdorff measure. To show it is a $J$-holomorphic $1$-subvariety, we only need to show that it has a positive cohomology assignment with $X=M_1$. 

For any admissible map $\sigma: D\rightarrow M_1$ with respect to $A=u^{-1}(Z_2)$, we assign an integer $IC(\sigma)$ as follows. When $\sigma: D\rightarrow M_1$ is admissible, its composition with the pseudoholomorphic map $u: M_1\rightarrow M$, $u\circ \sigma:D\rightarrow M$, is also admissible with respect to $Z_2\subset M$. There exists an arbitrarily small perturbation of $u\circ \sigma$ which produces a map $\sigma'$ homotopic to $u\circ \sigma$ through admissible maps such that $\sigma'$ is transverse to $Z_2$. This is called an admissible perturbation. Remark that we have to perturb the composition $u\circ \sigma$ instead  of just $\sigma$ to achieve transversality. The set $T$ of intersection points of $\sigma'(D)$ with $Z_2$ is a finite set of signed points. We define $IC(\sigma)$ to be the sum of these signs. By general intersection theory of submanifolds, see {\it e.g.} \cite{GP}, the intersection number $IC(\sigma)$ is independent of the choice of the admissible perturbation. 

\begin{prop}\label{icpca}
The assignment $IC(\sigma)$ to an admissible map $\sigma: D\rightarrow M_1$ defines a positive cohomology assignment to $A=u^{-1}(Z_2)$. 
\end{prop}
\begin{proof}

In our situation, Definition \ref{PCA}(1) means if $u\circ \sigma(D)\cap Z_2=\emptyset$, then $IC(\sigma)=0$. This is clear from our definition.

Assertion (2) of Definition \ref{PCA} follows from the following so-called Boundary Theorem \cite{GP}.
\begin{theorem}\label{bd}
Suppose $X$ is the boundary of some compact manifold $W$ and $g: X\rightarrow M$ is a smooth map.  If $g$ extends to all of $W$, then the intersection number of $g$ and $Z$ is zero for any closed submanifold $Z$ in $M$ of complementary dimension. 
\end{theorem}

Here we have $X=S^2$ and $W=D\times [0, 1]$, $g=\partial h$ and $Z=Z_2$. Moreover, our admissible maps are understood as their composition with the map of $u: M_1\rightarrow M$. Since $Z_2$ intersects the closure of $Image(h)$ inside $Image(h)$, we know the intersection number of $g$ and $Z_2$ is $IC(\sigma_0)-IC(\sigma_1)$. By Boundary Theorem, it is zero.

To show assertion (3), we first choose an admissible map $\sigma'$ (with respect to $Z_2$) transverse to $Z_2$ which is perturbed from $u\circ\sigma$. Hence $\sigma'^{-1}(Z_2)$ is a finite set of signed points in $D$. Since the degree of a map $f: X\rightarrow Y$ is just the intersection number of $f$ and any point $y\in Y$, we know $IC(\sigma'\circ\theta)$ is the sum of the signed points in $\sigma'^{-1}(Z_2)$ multiplied by $\deg\theta=k$. That is  $IC(\sigma\circ\theta)=k\cdot IC(\sigma)$.

For assertion (4), since $\sigma|_{D-\cup_iD_i}\cap A =\emptyset$, we can choose the admissible perturbation $\sigma'$ of $u\circ \sigma$ within $\cup_iD_i$. The intersection number $IC(\sigma)$ which is calculated as the  sum of signs of intersection points of $\sigma'$ is thus $\sum_iIC(\sigma\circ \theta_i)$.

When $\sigma$ is $J$-holomorphic, the composition $u\circ \sigma: D \rightarrow M$ is a $J$-holomorphic map. If it is non-constant, although it is not an embedding in general, we know $u\circ \sigma$ is an admissible map with respect to $Z_2$ since $\sigma(\partial D)\cap A=\emptyset$ if and only if $u\circ \sigma(\partial D)\cap Z_2=\emptyset$. Hence, the statement of assertion (5) for this case follows from the positivity of intersections of a $J$-holomorphic curve and an almost complex divisor, {\it i.e.} Lemma \ref{posint}. If $u\circ \sigma$ is a constant map, and since we assume $\sigma^{-1}(A)\ne \emptyset$, which is equivalent to $(u\circ\sigma)^{-1}(Z_2)\ne \emptyset$, we know the constant map $u\circ \sigma$ maps to a point in $Z_2$. Hence, we have $\sigma(D)\subset A$, which contradicts to the assumption that $\sigma$ is admissible. Hence, Definition \ref{PCA}(5) also follows. 
\end{proof}

Now we are ready to prove our first main result.
\begin{proof}(of Theorem \ref{ICdim4})
By Lemma \ref{compact} and Proposition \ref{Hd2}, $A=u^{-1}(Z_2)$ is a closed set with finite $2$-dimensional Hausdorff measure. By Proposition \ref{icpca}, $A$ could be endowed with a positive cohomology assignment, $IC(\sigma)$, for each admissible map $\sigma: D\rightarrow M_1$. Hence, by Proposition \ref{pcaholo}, the preimage $A=u^{-1}(Z_2)$ supports a $J_1$-holomorphic $1$-subvariety $\Theta$. 
\end{proof}

\subsection{When $Z_2$ is a $1$-subvariety in an almost complex $4$-manifold}\label{Z2subv}
When $(M, J)$ is an almost complex $4$-manifold, the statement of Theorem \ref{ICdim4} still holds even if $Z_2$ is merely assumed to be a $J$-holomorphic $1$-subvariety.  Eventually, what we need is a slightly more general intersection theory working for continuous maps and a version of Lemma \ref{posint} and Proposition \ref{uniquecont}. These are available in section 7 of \cite{MW}.

Let $\Sigma_i$ be compact oriented surfaces (with or without boundary), and let $u_i: \Sigma_i\rightarrow M$ be continuous maps such that $u_1(\partial\Sigma_1)\cap u_2(\Sigma_2)=\emptyset=u_2(\partial\Sigma_2)\cap u_1(\Sigma_1)$. Then by a generic smooth perturbation of $u_i$ relative to the boundary $\partial \Sigma_i$, we get maps $v_i$ such that if $v_1(p_1)=v_2(p_2)$, then $v_i$ is immersed at $p_i$ and the maps are transverse there, {\it i.e.} 
\begin{equation}\label{maptran}
T_{v_i(p_i)}=v_{1*}(T_{p_1}\Sigma_1)\oplus v_{2*}(T_{p_2}\Sigma_2).
\end{equation}
 Hence, the set of intersections of $v_1$ and $v_2$ is a finite set $T(v_1, v_2)$ of signed points $(p_1, p_2)$ such that $v_1(p_1)=v_2(p_2)$ where the sign $\delta(p_1, p_2)$ is $1$ if the left and right sides of \eqref{maptran} have the same orientation and $-1$ if not. The intersection number $u_1\cdot u_2$ is the sum of these signs. It is independent of the choice of the $v_i$, and thus is a homotopy invariant of $u_i$ relative to the boundaries. 
  
This intersection form can be localized as follows. Suppose $(p_1, p_2)$ is an isolated point of $T(u_1, u_2)$. Then there is some neighborhood $U$ of $P=u_i(p_i)$ such that if $W_i$ is the connected component of $u_i^{-1}(U)$ containing $p_i$, then $T(u_1|_{W_1}, u_2|_{W_2})$ contains only the point $(p_1, p_2)$. The intersection number $u_1|_{W_1}\cdot u_2|_{W_2}$ is independent of the choice of such $U$, and is thus a local invariant $\delta_{u_1, u_2}(p_1, p_2)$. Furthermore, if $T(u_1, u_2)$ is finite, then $u_1\cdot u_2=\sum_{u_1(p_1)=u_2(p_2)}\delta_{u_1, u_2}(p_1, p_2)$.
 
What we need to replace Lemma \ref{posint} and Proposition \ref{uniquecont} is the following, which is Theorem 7.1 of \cite{MW}.

\begin{theorem}\label{MW7.1}
Let $u_i: \Sigma_i\rightarrow M$ be maps that are $J$-holomorphic in a neighborhood of $p_i\in \Sigma_i$, where $u_1(p_1)=u_2(p_2)$. Suppose also that there are no neighborhoods $D_i$ of $p_i$ such that $u_1(D_1)=u_2(D_2)$. Then $(p_1, p_2)$ is an isolated intersection point of $T(u_1, u_2)$, and the intersection number $\delta_{u_1, u_2}(p_1, p_2)$ is greater than or equal to $1$, and strictly greater than $1$ unless $u_1$ and $u_2$ are transverse immersions at $p_1$ and $p_2$.
\end{theorem}

With Theorem \ref{MW7.1} in hand, our argument for Theorem \ref{Hd2} extends to the case when $(M, J)$ is an almost complex $4$-manifold and $Z_2$ is a $J$-holomorphic $1$-subvariety. In fact, the argument shows that the preimage under $u$ of each irreducible component of $Z_2$ has finite $2$-dimensional Hausdorff measure. 

For the second part, $IC(\sigma)$ is also well defined for any admissible map $\sigma: D\rightarrow M_1$ with respect to $A=u^{-1}(Z_2)$ by our discussion of intersection theory above. Moreover, Proposition \ref{icpca} still holds since Theorem \ref{bd} also holds when $Z$ is the image of a closed manifold of complementary dimension. Hence Theorem \ref{ICdim4} holds in our setting. 

\begin{theorem}\label{ICdim4MW}
Suppose $(M^{4}, J)$ is an almost complex $4$-manifold, and  $Z_2$ is a $J$-holomorphic $1$-subvariety. Let $(M_1, J_1)$ be a closed connected almost complex $4$-manifold and $u: M_1\rightarrow M$ a pseudoholomorphic map such that $u(M_1)\nsubseteq Z_2$. Then $u^{-1}(Z_2)$ supports a $J_1$-holomorphic $1$-subvariety in $M_1$.
\end{theorem}

\subsection{The homology class}\label{homology}
We can also determine the homology class of the $J$-holomorphic $1$-subvariety $A\subset M_1$ (and its image $u(A)\subset M$) by intersection pairing. 
The homology class is in fact determined by the positive cohomology assignment associated to $A$. First, given a $J$-holomorphic $1$-subvariety $\Theta=\{(C_i, m_i)\}$, there is a positive cohomology assignment for its support $C=|\Theta|=\cup C_i$. 
Let $C_i=\phi_i(\Sigma_i)$ where each $\Sigma_i$ is a compact connected complex curve and $\phi_i: \Sigma_i\rightarrow X$ is a $J$-holomorphic map embedding off a finite set. When $\sigma: D\rightarrow X$ is admissible, there is an arbitrarily small perturbation, $\sigma'$, of $\sigma$ which is homotopic to $\sigma$ through admissible maps and it is transverse to each $\phi_i$. Each fiber product $T_i:=\{(x, y)\in D\times \Sigma_i| \sigma'(x)=\phi_i(y)\}$ is a finite set of signed points of $D\times \Sigma$. We associate a weight $m_i$ to each signed point in $T_i$. The weighted sum of these signs in $\cup T_i$ is a positive cohomology assignment, denoted by $IS_{\Theta}$. 

Conversely, once a positive cohomology assignment $I$ is given as in Proposition \ref{pcaholo} and $C=\cup C_i$, we can associate the positive weight $m_i$ to $C_i$ as $I(\sigma)$ where $\sigma$ is a $J$-holomorphic disk intersecting transversey to $C_i$ at a smooth point. For the subvariety $\Theta=\{(C_i, m_i)\}$ obtained in this way, we have $I=IS_{\Theta}$. 

In our situation, the above construction gives rise to a $J$-holomorphic $1$-subvariety $\Theta$ of $M_1$ such that $|\Theta|=A=u^{-1}(Z_2)$ and $IS_{\Theta}(\sigma)=IC(\sigma)$. This subvariety will be called an {\it intersection subvariety} associated to $u$ later. Since  any homology class $\xi \in H_2(M_1, \mathbb Z)$ is representable by an embedded submanifold, the above claim just implies $\xi\cdot e_{\Theta}=u_*(\xi) \cdot [Z_2]$ as integers. Here $u_*(\xi)$ denotes the induced class in Borel-Moore homology $H_2^{BM}(M)$ and the latter product is understood as the intersection paring in Borel-Moore homology. The homology class $e_{\Theta}$ is determined by the intersection pairing with all the classes in $H_2(M_1, \mathbb Z)$. Since the latter product is determined only by the homology class of $Z_2$ and the homotopy class of the map $u$, we know $e_{\Theta}$ is the same as the homology class of the submanifold that is obtained by a perturbation of $u$ which is transverse to $Z_2$.

There are two important special cases. In the first, we assume the ambient manifold $M$ is closed.
\begin{prop}\label{closedM}
When $M^{2n}$ is a closed almost complex manifold, and $Z_1, Z_2$ are compact connected complex submanifolds of dimension $4$ and $2n-2$ respectively. Then the intersection $Z_1\cap Z_2$ is either one of $Z_i$, or supports a $J$-holomorphic $1$-subvariety of class $PD_{Z_1}^{-1}(\iota_{Z_1}^*(PD_M[Z_1]\cup PD_M[Z_2]))$ in $Z_1$ (and $PD_M^{-1}(PD_M[Z_1]\cup PD_M[Z_2])$ in $M$).
\end{prop}
This is just a refined version of Corollary \ref{eICdim4} in the introduction when $M$ is closed. The map $\iota_{Z_i}: Z_i\rightarrow M$ is  the inclusion. 
\begin{proof}
For the first statement, we apply Theorem \ref{ICdim4} to the embedding $u: M_1\rightarrow M$, such that $Z_1=u(M_1)$. 

The homology class is calculated by deforming $Z_2$ to $Z_2'$ such that $Z_1\pitchfork Z_2'$. We denote the intersection subvariety of $Z_1$ and $Z_2$ by $\Theta$. By the above discussion, $e_{\Theta}=[Z_1\cap Z_2]$. The homology of the latter in $M$ is $PD_M^{-1}(PD_M[Z_1]\cup PD_M[Z_2])$ following from standard intersection theory of submanifolds.

For the homology in $Z_1$, it is determined by intersection pairing with all classes in $H_2(Z_1, \mathbb Z)$. For any $a\in H_2(Z_1, \mathbb Z)$, $$PD_{Z_1}(a)\cup \iota_{Z_1}^*(PD_M[Z_1]\cup PD_M[Z_2])=PD_M((\iota_{Z_1})_*a)\cup PD_M[Z_1]\cup PD_M[Z_2].$$
Hence the conclusion follows. 
\end{proof}

In particular, when one of the $[Z_i]$ equals zero, then the homology class of the intersection subvariety is zero. On the other hand, when $J$ is tamed by a symplectic form in $M$, then the homology class (in $M$ or $Z_1$) of the intersection subvariety is non-trivial. 

The other important special case of Theorem \ref{ICdim4} is the application to complex line bundles over a $4$-dimensional almost complex manifold. This is contained in Section \ref{Jsec}.

In the following, we show that Proposition \ref{closedM} could also be stated for our general settings of Theorem \ref{ICdim4}, although this calculation will not be used later in this paper. Again, we only need to calculate the homology class of intersection when $u$ is transverse to $Z_2$.  

In general, our ambient manifold $M$ is not assumed to be compact. Hence, our discussion will be under the Borel-Moore homology framework. For an overview, see \cite{Ful}. Borel-Moore homology could be defined using singular cohomology. If a space $X$ is embedded as a closed subspace of $\mathbb R^n$, then $$H_i^{BM}(X, \mathbb Z)=H^{n-i}(\mathbb R^n, \mathbb R^n\setminus X).$$ 

Each Borel-Moore $i$-cycle $C$ (and in turn its homology class) determines a linear map $H_c^i(X, \mathbb Z)\rightarrow \mathbb Z$. If $M$ is an oriented, connected, real $n$-manifold, then $H_n^{BM}(M, \mathbb Z)$ is freely generated by a fundamental class $[M]$. The Poincar\'e dual of the cycle $C$ is the cohomology class $\eta_C^M\in H^{m-k}(M, \mathbb Z)$ uniquely determined by the equality $$\int_M a\wedge \eta_C^M=C(a), \, \, \forall a\in H_c^k(M, \mathbb Z).$$ 

A closed oriented submanifold $Z\subset M$ of dimension $k$ defines a $k$-dimensional BM cycle $[Z]$ 
$$[Z](a):=\int_Z a, \, \, \forall a\in H_c^k(M).$$ The Poincar\'e dual is the cohomology class $\eta_S^M\in H^{m-k}(M)$ which is uniquely determined by $$\int_M a\wedge \eta_Z^M=\int_Z a, \, \, \forall a\in H_c^k(M).$$

When $M$ is a closed manifold, $\eta_S^M$ is just the usual Poincar\'e dual class in singular cohomology. If $M$ is the total space of an oriented vector bundle $E$ over $S$, then $\eta_S^M$ is the Thom class of $E$. 

Suppose we have an oriented closed submanifold $Z_2\subset M$, and a smooth map $u: M_1\rightarrow M$ where $M_1$ is a compact manifold. If $u\pitchfork Z_2$, we have $\eta_{u(M_1)\cap Z_2}^M=\eta_{u_*([M_1])}^M\cup \eta_{Z_2}^M$, and $\eta_{u^{-1}(Z_2)}^{M_1}=u^*(\eta_{u_*([M_1])}^M\cup \eta_{Z_2}^M)$.

\subsection{Higher dimensions}\label{dim>4}
Our argument for Theorem \ref{ICdim4} could also be applied to other cases with no excess intersection phenomenon, {\it i.e.} when $Z_2$ is a codimension $2$ compact connected almost complex submanifold in $(M, J)$ and $u: M_1\rightarrow M$ is  a pseudoholomorphic map  such that $u(M_1)\nsubseteq Z_2$ and $\dim_{\mathbb R} M_1>4$. We are able to prove the following

\begin{theorem}\label{ICdim>4}
Suppose $(M^{2n}, J)$ is an almost complex $2n$-dimensional manifold, and  $Z_2$ is a codimension $2$ compact connected almost complex submanifold. Let $(M_1, J_1)$ be a compact connected almost complex manifold of dimension $2k<2n$ and $u: M_1\rightarrow M$ a pseudoholomorphic map such that $u(M_1)\nsubseteq Z_2$. Then $u^{-1}(Z_2)\subset M_1$ is a closed set with finite $(2k-2)$-dimensional Hausdorff measure and a positive cohomology assignment.
\end{theorem}

The following question asks for a generalization of Proposition \ref{pcaholo}. 
\begin{question}\label{pcaholoco2}
Let $(X, J)$ be a $2k$-dimensional almost complex manifold and let $C\subset X$ be a closed set with finite $(2k-2)$-dimensional Hausdorff measure and a positive cohomology assignment. Does $C$ support a compact $J$-holomorphic subvariety of complex dimension $k-1$?
\end{question}

If the answer to Question \ref{pcaholoco2} is affirmative, then we know the set $u^{-1}(Z_2)$ in Theorem \ref{ICdim>4} is a $J_1$-holomorphic subvariety. 

The proof of Theorem \ref{ICdim>4} is almost identical to that of Theorem \ref{ICdim4}. However, we need the following lemma, whose proof closely follows from arguments in section 5 of \cite{T}. A similar result can be found in \cite{Ush}. For completeness, we include its proof.

\begin{lemma}\label{diskfol}
Let $J_1$ be an almost complex structure on $\mathbb C^n$ which agrees with the standard almost complex structure $J_0$ at the origin. Choose an almost Hermitian metric $g$ compatible with $J_1$. There exists a constant $\rho_0$ with the following property. Let $\rho<\rho_0$ and let $U$ be the ball of radius $\rho$ in $\mathbb C^{n-1}$ and $D\subset \mathbb C$ the disk of radius $\rho$. Then there is a diffeomorphism $f: D\times U\rightarrow \mathbb C^n$, and constants $L, L_m$ depending only on $g$ and $J_1$, such that
\begin{itemize}
 \item For all $w\in U$, $f(D_{w})$ is a $J_1$-holomorphic submanifold containing $(0, w)$. 
\item For all $w\in U$, dist$((\xi, w); f(\xi, w))\le L\cdot \rho\cdot |\xi|$. 
\item For all $w\in U$, the derivatives of order $m$ of $f$ are bounded by $L_m\cdot \rho$. 
\item For any $\kappa\in \mathbb CP^{n-1}$, we can choose $f(D_{0})$ such that it is tangent at the origin to the line $l_{\kappa}\subset \mathbb C^n$ determined by $\kappa$.
\end{itemize}
\end{lemma}
\begin{proof}
We search for $J_1$-holomorphic disks which are perturbations of a $J_0$-holomorphic disk $(\xi, w_1+\kappa_1\xi, \cdots, w_{n-1}+\kappa_{n-1}\xi)$ where $w=(w_1, \cdots, w_{n-1})\in D^{n-1}$ and $\kappa=[1:\kappa_1:\cdots : \kappa_{n-1}]$. These disks could be expressed as $$q_{w, \kappa}(\xi)=(\xi, w_1+\kappa_1\xi+\tau_1(w, \kappa, \xi), \cdots, w_{n-1}+\kappa_{n-1}\xi+\tau_{n-1}(w, \kappa, \xi)),$$ whose $J_1$-holomorphic equations are $$\frac{\partial\tau_i}{\partial \bar \xi}=Q_i(w, \kappa, \tau_i(w, \kappa, \xi), \cdots, \tau_{n-1}(w, \kappa, \xi)),$$ such that 
\begin{equation}\label{Qsmall}
||Q_i||_{C^k}\le C_k||J_1-J_0||_{C^k(D^n)}.
\end{equation}

Now introduce a cutoff function $\chi_{\rho}:\mathbb C\rightarrow [0, 1]$ which equals $1$ for $|\xi|<\rho$ and $0$ for $|\xi|>\frac{3\rho}{2}$, and search for a solution to $$\frac{\partial\tau_i}{\partial \bar \xi}=\chi_{\rho}Q_i, \,\,\, i=1, \cdots, n-1.$$ The search is on the class of $(n-1)$-tuples of $C^{2, \frac{1}{2}}$ functions $\tau_i$ restricting to the circle of radius $4\rho$ around zero in the span of functions $\{e^{ik\theta}|k<0\}$, obeying $$\tau_i(\xi)=\frac{1}{\pi}\int\frac{\chi_{\rho}Q_i(w, \kappa, \tau_i(c, \kappa, \xi))}{\xi-\eta}d^2\eta, \, \, i=1, \cdots, n-1.$$ This class of functions is a Banach space using the following norm for $\tau=(\tau_1, \cdots, \tau_{n-1})$ (it is the $(n-1)$-fold direct sum of the norm used on page 886 of \cite{T}) $$||\tau||=\sum_{i=1}^{n-1}\sup_{t, s\in \mathbb C}(|\tau_i|+\rho\cdot |d\tau_i|+\rho^2\cdot |\nabla d\tau_i|+\rho^{\frac{5}{2}}\cdot \frac{|\nabla d(\tau_i)_t-\nabla d(\tau_i)_s|}{|t-s|^{\frac{1}{2}}}).$$  Applying the contractive mapping theorem to this Banach space, thanks to inequality \eqref{Qsmall}, as in Lemma 5.5 of \cite{T} the solution varies smoothly in each of $\xi, c$ and $\kappa$, and satisfies the bounds $$|\frac{\partial \tau}{\partial w_i}|<C\rho, \, \, |\frac{\partial \tau}{\partial \kappa_i}|<C\rho^2, \, \, ||\tau||_{C^0}<C(\rho^2+\rho(|w|+|\kappa|)), $$ $$||\tau||_{C^1}<C(\rho+|w|+|\kappa|).$$

Then  the lemma follows from the above discussion for a constant $\kappa$. Indeed, there exists $\epsilon>0$ with the property that when $|w|<\epsilon$, there is a unique small solution $\tau_w$ for the given constant $\kappa$. The corresponding map $q_w:=q_{w, \kappa}$ is then pseudoholomorphic. As the pair $(\xi, w)\in \mathbb C\times \mathbb C^{n-1}$ vary, $f(\xi, w):=q_w(\xi)$ defines a map from a neighborhood of the origin in $\mathbb C^n$ to $\mathbb C^n$. The implicit function theorem asserts that $f$ is a diffeomorphism on some neighborhood of $0\in \mathbb C^n$ if its differential at $0$ is invertible. This will be the case if $|\frac{\partial \tau_w}{\partial w}|<1$ at $(\xi, w)=0$.  The latter inequality is insured when $\rho$ is small. 

The above argument is for the affine plane $[1:\kappa_1: \cdots :\kappa_{n-1}]$, but certainly it works also for other affine planes $[\kappa_1: \cdots: 1: \cdots: \kappa_{n-1}]$. 
\end{proof}

Then Theorem \ref{ICdim>4} follows from the same argument as for Theorem \ref{ICdim4}.

\begin{proof}(of Theorem \ref{ICdim>4}). 
By Lemma \ref{compact}, $A=u^{-1}(Z_2)$ is a closed set. 

To show the $(2k-2)$-dimensional Hausdorff measure $\mathcal H^{2k-2}(A)$ is finite, we follow the argument of Proposition \ref{Hd2}, but instead using Lemma \ref{diskfol}. First, for any point $x\in M_1$, the complex directions of $T_xM_1$ are parametrized by $\mathbb CP^{k-1}$. We choose a Gaussian normal coordinate such that a neighborhood of $x$ in $M_1$ is identified with a neighborhood of the origin in $\mathbb C^k$ and we are in the situation of Lemma \ref{diskfol}. Now we want to find a suitable complex direction $\kappa$, such that none of the $J_1$-holomorphic disks $f(D_{w})$ in Lemma \ref{diskfol} are contained in $u^{-1}(Z_2)$. 

If we take $w=0$ and vary $\kappa$ in the proof of Lemma \ref{diskfol}, our construction would provide a smooth map $f_0: D\times U\rightarrow \mathbb C^k$, such that each $f_0|_{D_{\kappa}}$, $\kappa\in U\subset \mathbb CP^{k-1}$, is an embedding whose image is a $J_1$-holomorphic disk which is tangent at the origin to the line $l_{\kappa}\subset \mathbb C^k$ determined by $\kappa$. Moreover, $f_0$ maps the zero section $\{0\}\times U$ to $0\in \mathbb C^n$ and $f_0|_{(D\setminus \{0\})\times U}$ is a diffeomorphism onto its image by implicit function theorem. This is essentially Lemma 6.1 in \cite{Ush}.

Now, for some $\kappa\in U$, $f_0(D_{\kappa})$ are not in $u^{-1}(Z_2)$, otherwise the open set $f_0(D\times U)$ is contained in $u^{-1}(Z_2)$, which contradicts  our assumption that $u(M_1)\nsubseteq Z_2$ by Proposition \ref{uniquecont}. Moreover, for this $\kappa$, $f_0(D_{\kappa})\cap u^{-1}(Z_2)$ is a finite set by Lemma \ref{posint}. Then we choose this $\kappa$ to construct our $f$ in Lemma \ref{diskfol}. For $\rho$ small enough, it gives a diffeomorphism $f: D\times U\rightarrow \mathbb C^k$ such that for each $w\in U$, $f(D_{w})$ intersects  $u^{-1}(Z_2)$ only at finitely many points. Thus by the same coarea formula argument as in Proposition \ref{Hd2}, we know $\mathcal H^{2k-2}(A)<\infty$. 

Finally, the construction $IC(\sigma)$ of Proposition \ref{icpca} again defines a positive cohomology assignment to $A=u^{-1}(Z_2)$.
\end{proof}

\section{Pseudoholomorphic sections of complex line bundles}\label{Jsec}
 Let $(M, J)$ be an almost complex manifold and $\pi: E\rightarrow M$ a complex vector bundle over it. In most of our discussions in this section, $\dim_{\mathbb R} M=4$ and $E$ is a complex line bundle. If the total space of $E$ is endowed with an  almost complex structure $\mathcal J$, a $(J, \mathcal J)$-holomorphic (or pseudoholomorphic) section of the bundle $E$ is a smooth map $s: M\rightarrow E$ such that $\pi\circ s=id_M$ and $\mathcal Jds=dsJ$. Equivalently, $s$ is a smooth section of $E$ such that the image $s(M)$ is a $\mathcal J$-holomorphic submanifold.

We know the set of equivalence classes of complex line bundles over $M$ is the same as $H^2(M, \mathbb Z)$. And the total space of the bundle could be associated with different almost complex structures. See the later discussion for the canonical bundle. If $J$ is a complex structure and $E$ is a holomorphic line bundle, then the zero locus of any holomorphic section is a divisor. Moreover, for every divisor $D\subset M$, there is a holomorphic line bundle $E$ over $M$ and a holomorphic section $s$ such that $s^{-1}(0)=D$. The following is what we know for almost complex $4$-manifolds.  
\begin{prop}\label{Hbun}
Suppose $(M, J)$ is a closed almost complex $4$-manifold. 
Then the zero locus of any nontrivial pseudoholomorphic section of any complex line bundle $E$ with any almost complex structure on the total space extending $J$ supports a $J$-holomorphic subvariety in $(M, J)$ in class $PD(c_1(E))$. The subvariety is uniquely determined. 
\end{prop}
We call such a subvariety the zero divisor of the section $s$. 
\begin{proof}
Suppose the total space of the complex line bundle $E$ admits an almost complex structure $\mathcal J$ such that $\mathcal J|_M=J$. The images of the zero section $0(M)$ and nontrivial section $s(M)$ are both $\mathcal J$-holomorphic submanifolds. Hence, by Theorem \ref{ICdim4}, their intersection which is the zero locus $|s^{-1}(0)|$ supports a $J$-holomorphic subvariety of $M$.
 As explained in section \ref{homology}, such a $J$-holomorphic subvariety could be chosen in the homology class of the zero locus of a smooth transverse section. This homology class is Poincar\'e dual to $c_1(E)$.
 
Moreover, this subvariety is uniquely determined as argued in Section \ref{homology}. That is because the positive cohomology assignment given by the sections $s(M)$ and $0(M)$ determines the subvariety supported on $|s^{-1}(0)|$.
\end{proof}

\begin{cor}\label{nosec}
If $(M^4, J)$ does not have $J$-holomorphic subvarieties in a class $a\in H_2(M, \mathbb Z)$, then any complex line bundle whose Chern class is $PD(a)$ does not admit nontrivial pseudoholomorphic sections for any almost  complex structure on the total space whose restriction to $M$ is $J$.
\end{cor}
\begin{proof}
By Proposition \ref{Hbun}, there is a $J$-holomorphic subvariety representing the homology class $PD(c_1(E))$ when there is a  nontrivial pseudoholomorphic section.
\end{proof}

A generic tamed almost complex structure on a $4$-torus or a K3 surface does not have any pseudoholomorphic curves. In fact, a generic almost complex structure on a $4$-manifold $(M, J)$ does not admit any non-constant pseudoholomorphic function even locally. Hence, a generic pseudoholomorphic $1$-subvariety is not the zero locus of any complex line bundle over $M$. Since  pseudoholomorphic $1$-subvarieties in a $4$-dimensional (tamed) almost complex manifold $(M, J)$ are the generalization of the notion of Weil divisor, it just says that, in general, a Weil divisor is not a Cartier divisor in an almost complex manifold.

In the following, we study in detail a particularly interesting line bundle, the canonical bundle. When the base dimension is $4$, the suitable generalization of canonical bundle to the almost complex setting is the bundle of $J$-anti-invariant $2$-forms $\Lambda_J^-$. We recall the definition. The almost complex structure acts on the bundle of real 2-forms $\Lambda^2$ as an involution, by $\alpha(\cdot, \cdot) \rightarrow \alpha(J\cdot,
J\cdot)$. This involution induces the splitting into $J$-invariant, respectively,
$J$-anti-invariant 2-forms
$$\Lambda^2=\Lambda_J^+\oplus \Lambda_J^-$$
corresponding to the eigenspaces of eigenvalues $\pm 1$ respectively.

The bundle $\Lambda_J^-$ has (real) rank $2$. It inherits a complex structure, also denoted by $J$, given by $J\phi(X, Y)=-\phi(JX, Y)$. Hence $\Lambda_J^-$ is a complex line bundle over $M$. Moreover, we can calculate its Chern class.

\begin{prop}\label{c1=k}
The first Chern class of the complex line bundle $\Lambda_J^-$ over $(M^4, J)$ is the canonical class $K_J$ of the almost manifold $(M, J)$.
\end{prop}
\begin{proof}
The conclusion should be well known. We offer a proof by Chern-Weil theory. The calculation is useful for later discussions.

Let $g$ be a Riemannian metric compatible with $J$, {\it i.e.} $g(JX, JY)=g(X, Y)$. We call the triple $(M, J, g)$ an almost Hermitian manifold. Denote the complexified tangent space by $T^{\mathbb C}M=TM\otimes \mathbb C$. The complexified tangent space can be decomposed by as $T^{\mathbb C}M=T'M\oplus T''M$ where $T'M$ and $T''M$ are the eigenspaces of $J$ corresponding to eigenvalues $\sqrt{-1}$ and $-\sqrt{-1}$ respectively.  Choose a local unitary frame $\{e_1, e_2\}$ for $T'M$ with respect to the Hermitian inner product induced from $g$. 

We choose an  almost Hermitian connection $\nabla$ on $(M, J, g)$, {\it i.e.} $\nabla J=\nabla g=0$, which exists on every almost Hermitian manifold. There is a matrix of complex valued $1$-forms $\{\theta_i^j\}$, called the connection $1$-forms, such that $\nabla e_i=\theta_i^j e_j$. The curvature $\Omega=\{\Omega_i^j\}$ is defined by $$\Omega_i^j=d\theta_i^j+\theta_k^j\wedge \theta_i^k.$$ By Chern-Weil theory $c_1(TM, J)=[\frac{\sqrt{-1}}{2\pi}tr(\Omega)]$. In local coordinates, $$tr(\Omega)=d\theta_1^1-\theta_2^1\wedge\theta_1^2+d\theta_2^2-\theta_1^2\wedge\theta_2^1=d\theta_1^1+d\theta_2^2.$$

On the other hand, the almost Hermitian connection $\nabla$ induces a connection on $\Lambda_J^-$, which will also be denoted by $\nabla$. Again, we consider the complexification of $\Lambda_J^-$ and look at its eigenspace corresponding to the eigenvalue $\sqrt{-1}$ with respect to the induced involution $J: \Lambda_J^-\rightarrow \Lambda_J^-$, $\phi\mapsto J\phi$. The complexified bundle of the dual of $\Lambda_J^-$ has local generator $e_1\wedge e_2$. We have $$\nabla(e_1\wedge e_2)=(\theta_1^1+\theta_2^2)e_1\wedge e_2.$$ Hence the curvature is $$d(\theta_1^1+\theta_2^2)+(\theta_1^1+\theta_2^2)\wedge (\theta_1^1+\theta_2^2)=d(\theta_1^1+\theta_2^2).$$
This is identical to $tr(\Omega)$. This implies the first Chern class of the dual bundle of $\Lambda_J^-$ is $c_1(TM, J)$. Hence, the first Chern class of $\Lambda_J^-$ is $K_J=-c_1(TM, J)$. 
\end{proof}

The following material partially arises from discussions with Tedi Draghici. 
As in the above proof, an almost Hermitian connection $\nabla$ on $(M, J)$ induces a connection on $\Lambda_J^-$. This in turn gives a decomposition of the tangent bundle of total space $\Lambda_J^-$
$$T\Lambda_J^- = \mathcal{H} \oplus \mathcal{V} \; .$$
If $\pi: \Lambda_J^- \rightarrow M$ denotes the projection, $ \mathcal{V} = \ker(d\pi)$ is identified with the fiber of the bundle and, restricted to $\mathcal{H}$, $d\pi$ is an isomorphism
between $\mathcal{H}$ and $TM$ which enables us to identify these spaces. From $J$ and $\nabla$, we also get an almost complex structure, $\mathcal{J}_{\nabla}$ on the
total tangent space of the bundle, where $\mathcal{J}_{\nabla} = J$ on $\mathcal{H} = TM$ and $\mathcal{J}_{\nabla} = - J$ on $\mathcal{V}$. For all such $\mathcal J_{\nabla}$, the first Chern class $c_1(T\Lambda_J^-, \mathcal J_{\nabla})$ of the total space is $0$. Hence the total space of $\Lambda_J^-$ is an open almost complex Calabi-Yau $6$-manifold. Moreover if $M$ is symplectic, then the total space $\Lambda_J^-$ admits a symplectic form. It follows from a general fact that the total space of any complex line bundle $L$ over a symplectic manifold $(M, \omega)$ is symplectic, following Thurston's construction as in \cite{Gom}. In fact, it directly follows from Theorem 2.1 in \cite{Gom} by first compactifying $L$ to a $\mathbb CP^1$ bundle and noticing an $S^2$ fiber cannot be nullhomologous by adjunction formula.  We can choose the symplectic form such that it tames $\mathcal J$. Hence, the total space of $\Lambda_J^-$ is an open symplectic Calabi-Yau $6$-manifold.

A section $\psi \in \Gamma(\Lambda_J^-)$ is $(J, \mathcal{J}_{\nabla})$-holomorphic if
$$\mathcal{J}_{\nabla}\circ D \psi = D\psi \circ J ,$$
where $D\psi : TM \rightarrow T\Lambda_J^-$ is the differential of $\psi$ as a section. It should not be confused with taking the exterior derivative when $\psi$ is treated as a $2$-form.

Since $\pi \circ \psi = id_M$, with the identifications above, the condition that $\psi$ is
$(J, \mathcal{J}_{\nabla})$-holomorphic is equivalent to
\begin{equation} \label{Jholo} 
(\nabla_{JX} \psi)(JY, Z) + (\nabla_X \psi)(Y,Z) = 0 , \; \; \forall X,Y,Z \in TM \; .
\end{equation}
It is obvious that $\psi$ is $(J, \mathcal{J}_{\nabla})$-holomorphic if and only if $J\psi$ is
$(J, \mathcal{J}_{\nabla})$-holomorphic since $\nabla J=0$. Notice that \eqref{Jholo} implies that our $\mathcal J_{\nabla}$ is a ``bundle almost complex structure" in the sense of \cite{dBT}. For a complex vector bundle $(E, \mathcal J)$ over $(M, J)$, in addition to requiring the bundle projection is $(\mathcal J, J)$-holomorphic and $\mathcal J$ inducing the standard complex structure on the fibers, a bundle almost complex structure also requires that the fiberwise addition $\alpha: E\times_M E\rightarrow E$ and the fiberwise multiplication by a complex number $\mu: \mathbb C\times E\rightarrow E$ are both pseudoholomorphic. 

Different almost Hermitian connections $\nabla$ give rise to different almost complex structures on $\Lambda_J^-$. Some of them might admit pseudoholomorphic sections. For a tamed $J$ there is always a $J$-holomorphic subvariety in the class $K_J$ if $b^+>1$ \cite{T}. However, in general it is not the zero divisor of a pseudoholomorphic section of the complex line bundle $\Lambda_J^-$ with the almost complex structure $\mathcal J_{\nabla}$ induced by an almost Hermitian connection $\nabla$.

If we take $\nabla$ to be the Chern connection (sometimes referred to as the second canonical connection), {\it i.e.} the unique almost Hermitian connection whose $(1,1)$ part of the torsion vanishes, Tedi Draghici observed that $\psi$ is $(J, \mathcal{J}_{\nabla})$-holomorphic if and only if $\bar{\partial} (\psi +iJ\psi) = 0$ where $\bar{\partial}(\psi +iJ\psi)$ denotes the projection on $\Lambda^{2,1}_J$ of the differential $d(\psi +iJ\psi)$. Notice that $d(\psi +iJ\psi)=0$ if and only if $J$ is integrable. 

On the other hand, it was conjectured in \cite{DLZiccm} that the zero set of a closed $J$-anti-invariant form $\alpha$ is $J$-holomorphic for any almost complex structure, see Question \ref{Janti0holo} in our introduction. However, this conjecture does not follow from our current framework, since it would imply the form $J\alpha$ is also closed, which would imply the integrability of $J$. This conjecture is confirmed in \cite{BonZ}. It applies the general strategy of proving Theorem \ref{ICdim4}, but uses a different analysis for the local model.

Finally, we would like to know whether the divisor determines the pseudoholomorphic section up to scaling.

\begin{prop}\label{Dtos}
Suppose $J$ is tamed. For a given almost Hermitian connection $\nabla$ on $\Lambda_J^-$, all the pseudoholomorphic sections of the bundle $\Lambda_J^-$ with the same zero divisor are linearly parametrized by $\mathbb C$.  
\end{prop}
\begin{proof}
Any smooth section of the bundle $\Lambda_J^-$ is a $J$-anti-invariant form. If $\psi$ is a form such that the corresponding section of $\Lambda_J^-$ is $(J, \mathcal J_{\nabla})$-holomorphic, then $(a+bJ)\psi$, $a, b\in \mathbb R$, are also pseudoholomophic sections by Equation \eqref{Jholo} and has the same zero divisor as $\psi$. 

There are no more pseudoholomorphic sections with divisor $\psi^{-1}(0)$. If there is any, say $\phi$, then at some point $x\in M$ which is not in  $|\psi^{-1}(0)|$, $\phi |_x=(a+bJ)\psi |_x$ for some constant $a+ib\in \mathbb C$. In other words, the intersection of $\mathcal J_{\nabla}$-holomorphic subvarieties corresponding to $\phi$ and $(a+bJ)\psi$ contains $\psi^{-1}(0)\cup \{x\}$. In particular, the homology class of intersection is different from $K_J$ since any $J$-holomorphic subvariety has non-trivial homology when $J$ is tamed. This contradicts to Theorem \ref{ICdim4} since all the sections have the same normal bundle which is $\Lambda_J^-$ and their intersections have homology class $K_J=[\psi^{-1}(0)]$.
\end{proof}

Proposition \ref{Dtos} also works for any complex line bundle $E$ over $M$ with bundle almost complex structure in the sense of \cite{dBT}.

In particular, the above discussion of $(J, \mathcal{J})$-holomorphic sections of $\Lambda_J^-$ could be applied to a similar study of the tensor complex line bundle $(\Lambda_J^-)^{\otimes_{\mathbb C} N}$. This would lead to several potential applications. For example, we can study the growth rate of certain $(J, \mathcal{J}_{\nabla})$-holomorphic sections of $(\Lambda_J^-)^{\otimes_{\mathbb C} N}$. Here, for each positive integer $N$, we record the maximal dimension of such sections for any almost complex structure $J$ tamed by the symplectic form $\omega$ and any $\nabla$. This method may generalize the definition of Kodaira dimension to symplectic manifolds. So far, such definition is known only for symplectic manifolds of dimensions no greater than $4$.

In summary, the current section describes a way to find $J$-holomorphic curves through almost Hermitian connections. Given a complex line bundle $E$ over an almost complex $4$-manifold $M$, any almost Hermitian connection of $E$ would induce an almost complex structure on the total space. Then the zero divisor of any pseudoholomorphic section is a $J$-holomorphic subvariety of $M$ in class $c_1(E)$. Of particular interest is the bundle $\Lambda_J^-$, which is a generalization of the canonical bundle in almost complex setting. To produce $J$-holomorphic subvarieties in class $K_J$, we look at the pseudoholomorphic sections for almost complex structures on $\Lambda_J^-$ induced by almost Hermitian connections on $M$. Then each such section is a $J$-anti-invariant form satisfying certain ``closedness" condition. The zero divisor of each such form is a $J$-holomorphic subvariety in class $K_J$. Notice, there is no tameness requirement for our almost complex structure $J$, while tameness of $J$ is essential in \cite{T} to guarantee a $J$-holomorphic subvariety in class $K_J$.

\begin{remark}
The bundle $\Lambda_J^-$ is the bundle of real part of $(2, 0)$ forms \cite{LZ}, {\it i.e.} $\Lambda_J^-=(\Lambda_J^{2, 0}(M)_{\mathbb C}\oplus\Lambda_J^{0, 2}(M)_{\mathbb C})\cap \Lambda^2(M)$. From this viewpoint, the corresponding canonical line bundle for an almost complex $2n$-manifold is the bundle of real part of $(n, 0)$ forms. Except lacking of the corresponding statement of Theorem \ref{ICdim4}, most results in this section still work for this complex line bundle. In particular, its first Chern class is $K_J$.
\end{remark}

\section{Pseudoholomorphic maps and symplectic birational geometry}\label{app}
In this section, we discuss applications of our results to pseudoholomorphic maps and symplectic birational geometry.

\subsection{Pseudoholomorphic maps}
As we remarked in the introduction, the image of pseudoholomorphic map $u$ in Theorem \ref{ICdim4} might not be of dimension $4$. This leads to the first application.

\begin{prop}\label{fibJhol}
Let $f: X\rightarrow S$ be a pseudoholomorphic map from a closed almost complex $4$-manifold $(X, J)$ to a closed Riemann surface $S$. Then for any $x\in S$, the preimage $f^{-1}(x)$ is a $J$-holomorphic $1$-subvariety of $(X, J)$.
\end{prop}
\begin{proof}
Apply Theorem \ref{ICdim4} to $M=S$ and $Z_2=\{x\}$, we know $f^{-1}(x)$ is a $J$-holomorphic $1$-subvariety in $(X, J)$.
\end{proof}

In the following, we will mainly study equi-dimensional pseudoholomorphic maps. The study of such a map is discussed in section 5 of \cite{Zadv}. We start with a direct geometric method.  The advantage of this approach is that we do not require any tameness of the almost complex structure of the domain. The second approach relies on a more delicate study of the singularity subset. Exploring it, we will answer Question 5.4 of \cite{Zadv}, see Theorem \ref{phblowdown}.

We start with the following simple well-known lemma. For readers' convenience, we include its proof.

\begin{lemma}\label{imm=deg>0}
Let $u: (X, J)\rightarrow (M, J_M)$ be a pseudoholomorphic map between closed connected almost complex $2n$-manifolds. Then the following are equivalent:

\begin{enumerate}
\item $u$ is somewhere immersed.
\item $\deg u\ne 0$.
\item $\deg u>0$.
\item $u$ is surjective.
\end{enumerate}
\end{lemma}
\begin{proof}
For any point $x\in X$, $\det(du_x)_{\mathbb R}=|\det(du_x)_{\mathbb C}|^2\ge 0$ since $u$ is pseudoholomorphic. Meanwhile, for any regular value $y\in M$ of the map $u$, the topological degree of the map $u$ is defined as $$\deg u:=\sum_{x\in u^{-1}(y)} \hbox{sign}(\det(du_x)_{\mathbb R}).$$ Hence, it is clear that $(2)\iff (3)$.

If $\deg u=0$, by Sard's theorem, a generic point of $M$ is a regular value. By the definition of topological degree, for any regular value $y$, $u^{-1}(y)$ is an empty set which implies $u$ is not surjective. This shows $(4)\implies (2)$.

It is a simple fact that a map $f: X^n\rightarrow Y^n$ of nonzero degree between closed orientable manifolds is surjective. Otherwise, if $y\notin \hbox{im}(f)$, the factorization $H_n(X)\rightarrow H_n(Y\setminus \{x\})\rightarrow H_n(Y)$ shows $\deg f=0$. Hence $(2)\implies (4)$.

If $u$ is immersed at $x$, since being immersed is an open condition, it is a local diffeomorphism on an open neighborhood $U_x$ of $x$. By Sard's theorem, a generic point of $M$ is a regular value, hence there is a regular value $y=u(x')$ for some $x' \in U_x$. Moreover, $\deg u\ge \hbox{sign}(\det(du_{x'})_{\mathbb R})>0$. Hence $(1)\implies (3)$. 

Finally, $\deg u\ne 0$ implies $u$ is immersed at regular points. This is $(2)\implies (1)$. 
\end{proof}
Later, we will state the equivalent conditions (1)-(3) alternatively. 

We are ready for our first approach. We only state and prove the result for the simplest target space, the projective plane. 
\begin{prop}\label{maptorr}
Let $u: X\rightarrow \mathbb CP^2$ be a non-constant pseudoholomorphic map from a closed connected almost complex $4$-manifold $(X, J)$ to $\mathbb CP^2$ with a tamed almost complex structure. Then $u^{-1}(x)$ is the union of a (possibly empty) $J$-holomorphic $1$-subvariety of $(X, J)$ and a set of finitely many points for  any $x\in \mathbb CP^2$.
\end{prop}
\begin{proof}

For any $x\in \mathbb CP^2$, we find two distinct  rational curves $S_x$ and $S_x'$ in class $H$ which pass through $x$. By the positivity of local intersections of pseudoholomorphic curves, $S_x$ and $S_x'$ only intersect at $x$ and they are not tangent to each other. We can assume $u(X)\nsubseteq S_x\cup S_{x}'$. 

By Lemma \ref{imm=deg>0},  $\deg u\ne 0$ if and only if $u$ is surjective. In fact, a generic point is not in the image of $u$ when $\deg u=0$. We choose $y\notin u(X)$ when $\deg u=0$ or any $y\in \mathbb CP^2$ when $\deg u\ne 0$, and choose $S_x$ to be the unique smooth rational $J$-holomorphic curve in the line class $H\in H^2(\mathbb CP^2, \mathbb Z)$ passing through $x$ and $y$. The existence is guaranteed by Taubes' SW=Gr \cite{T, LLb+1} where the Gromov-Taubes invariant of $H$ is $1$ and its Seiberg-Witten dimension is $H\cdot H+1=2$. We know it is irreducible because $H$ has the minimal symplectic energy and smooth because of adjunction formula. Then we choose another point $y'$ ($y'\notin u(X)\cup S_x$ if $\deg u=0$) and choose  $S_x'$ to be the smooth rational $J$-holomorphic curve in the line class $H\in H^2(\mathbb CP^2, \mathbb Z)$ passing through $x$ and $y'$.

By Theorem \ref{ICdim4}, $u^{-1}(S_x)$ and $u^{-1}(S_x')$ are both $J$-holomorphic subvarieties. Hence, their intersection $u^{-1}(S_x)\cap u^{-1}(S_x')=u^{-1}(x)$ is the union of a (possibly empty) $J$-holomorphic $1$-subvariety of $(X, J)$ and a set of finitely many points. 
\end{proof}

The statement of Proposition \ref{maptorr} is also true when $\mathbb CP^2$ is replaced by any symplectic $4$-manifold of $b^+=1$ with any tamed almost complex structure by applying results of \cite{LLwall}. For this generality, we need to use Theorem \ref{ICdim4MW} where $Z_2$ is merely a $J$-holomorphic $1$-subvariety.

Most almost complex manifolds do not admit non-constant pseudoholomorphic maps to $\mathbb CP^2$. The existence of such a map would require $X$ to contain many $J$-holomorphic curves, which is comparable to the requirement of the algebraic dimension $a(X)>0$ in the complex setting. As we can see from Proposition \ref{maptorr}, such an $X$ is covered by $J$-holomorphic subvarieties. 

\begin{cor}\label{curvecover}
Suppose there is a non-constant pseudoholomorphic map $u: X\rightarrow \mathbb CP^2$ from a closed connected almost complex $4$-manifold $(X, J)$ to $\mathbb CP^2$ with a tamed almost complex structure. Then there is a $J$-holomorphic subvariety passing through any point of $X$. 
\end{cor}
\begin{proof}
Choose any point $x\in X$. 
If $\deg u>0$, we choose a rational curve $S_x$ in class $H$ passing through $u(x)$. If $\deg u=0$, we choose this rational curve $S_x$ such that $u(X)\nsubseteq S_x$. 
Then preimage $u^{-1}(S_x)$ is a $J$-holomorphic curve passing through $x$ by Theorem \ref{ICdim4}.  
\end{proof}

In particular, a  $K3$ or $T^4$ with a generic tamed almost complex structure does not admit any pseudoholomorphic map to $\mathbb CP^2$ endowed with any tamed almost complex structure.

For our second approach to pseudoholomorphic maps between almost complex $4$-manifolds, we start with studying the singularity subset of a pseudoholomorphic map. For a pseudoholomorphic map $u: (X, J)\rightarrow (M, J_M)$, the singularity subset of $u$ is defined to be the points $p\in X$ such that the differential $du_p: T_pX\rightarrow T_{u(p)}M$ is not of full rank. 

\begin{theorem}\label{s4to4}
Let $u: (X, J)\rightarrow (M, J_M)$ be a somewhere immersed pseudoholomorphic map between closed almost complex $4$-manifolds. Then the singularity subset $\mathcal S_u$ of $u$ supports a $J$-holomorphic $1$-subvariety.
\end{theorem}
\begin{proof}
We first look at the manifold $J^1(X, M)$ of $1$-jets of pseudoholomorphic mappings from $(X, J)$ to $(M, J_M)$. It is identified with the total space of the complex vector  bundle $\mathcal E$ over $X\times M$, whose fiber is the complex vector space of all complex linear maps $L: T_xX\rightarrow T_mM$ regarding the almost complex structures $J|_x$ and $J_M|_m$, for each $x\in X$ and $m=u(x)\in M$. Therefore $\mathcal E$ is a complex vector bundle of rank $4$ in our situation. By \cite{Gau}, there is a canonical almost complex structure $\mathcal J$ on $J^1(X, M)$ such that for any pseudoholomorphic map $u: (X, J)\rightarrow (M, J_M)$, the canonical lift $u_{\mathcal E}(x)=(x, u(x),  (du_x)_{\mathbb C})$ is a pseudoholomorphic map from $(X, J)$ to $(J^1(X, M), \mathcal J)$. The conclusion could also be derived from the canonical almost complex structure on the total space of the tangent bundle of an almost complex manifold (\cite{YK}, see Theorem 3.2 in \cite{LS}), which would in turn give a canonical almost complex structure on $\mathcal E$. \footnote {These two canonical almost complex structures on $J^1(X, M)$ are different. As discussed in \cite{LS}, the one in \cite{Gau}, for the tangent bundle case, does not have the deformation property (d) in Theorem 3.2 of \cite{LS}, although we do not need this property here. } 

By taking the fiberwise complex determinant of the bundle $\mathcal E$, {\it i.e.}, replacing the fibers of $\mathcal E$, the complex linear maps $L: T_xX\rightarrow T_mM$, by the induced maps $\det L: \Lambda_{\mathbb C}^2T_xX \rightarrow \Lambda^2_{\mathbb C}T_mM$, we get a complex line bundle $\mathcal L$. Its total space inherits the canonical almost complex structure from $J^1(X, M)$. Its zero section $X\times M\times \{0\}$ is a codimension two $\mathcal J$-holomorphic submanifold. 
The map $u$ induces a map $u_{\mathcal L}(x)=(x, u(x), \det (du_x)_{\mathbb C})$ from $X$ to $\mathcal L$. It is a pseudoholomorphic map by the calculation of Proposition A.2.10 of \cite{Gau}. The intersection  $u_{\mathcal L}(X)\cap (X\times M\times \{0\})$ inside $\mathcal L$ is the set of points $(x, u(x), 0)$ where $u$ is not immersed at $x$. In particular, $u_{\mathcal L}^{-1}(X\times M\times \{0\})$ is the singularity subset  $\mathcal S_u$.

Apply Theorem \ref{ICdim4} to $Z_2= X\times M\times \{0\}$ and the ambient space $\mathcal L$ for the pseudoholomorphic map $u_{\mathcal L}$. Since $u$ is somewhere immersed, we know $u_{\mathcal L}(X)\nsubseteq Z_2$.  Hence, the singularity subset of $u$, $\mathcal S_u=u_{\mathcal L}^{-1}(X\times M\times \{0\})$, supports a $J$-holomorphic $1$-subvariety.
\end{proof}

The following is a quick corollary.

\begin{cor}\label{fromnc}
Let $u: (X, J) \rightarrow (M, J_M)$ be a degree non-zero pseudoholomorphic map between closed almost complex $4$-manifolds. If $X$ is simply connected with no $J$-holomorphic curves, {\it e.g.} a generic $K3$ surface, then $u$ is a diffeomorphism. If $X$ has no $J$-holomorphic curves, {\it e.g.} a generic $T^4$, then $u$ is a covering map.
\end{cor}
\begin{proof}
By Theorem \ref{s4to4} and the assumption that $X$ has no $J$-holomorphic curves, the singularity subset $\mathcal S_u=\emptyset$. Hence $u$ is a local diffeomorphism onto $M$. Since $X$ is compact, $u$ is a covering map. Moreover, when $X$ is simply connected, a covering map has to be a diffeomorphism. 
\end{proof}

The following lemma describes the structure of a somewhere immersed pseudoholomorphic map.

\begin{lemma}\label{decX}
For a somewhere immersed pseudoholomorphic map $u: (X, J)\rightarrow (M, J_M)$ between closed almost complex $4$-manifolds, the points of $X$ are divided into three groups:
\begin{enumerate}
\item the regular points $X\setminus \mathcal S_u$, where $u$ is a local diffeomorphism and the preimage of any point in $u(X\setminus \mathcal S_u)$ has at most $\deg u$ points in $X\setminus \mathcal S_u$;
\item a subset of singular points $\mathcal S_u'\subset \mathcal S_u$, consisting of all the irreducible components of $\mathcal S_u$ contracted by $u$;
\item points in $\mathcal S_u\setminus \mathcal S_u'$, where the preimage of each point in  $u(\mathcal S_u\setminus \mathcal S_u')$ has finitely many points for $u|_{\mathcal S_u\setminus \mathcal S_u'}$.
\end{enumerate}
\end{lemma}
\begin{proof}
Since the singularity subset $\mathcal S_u$ is $J$-holomorphic and $u$ is pseudoholomorphic, its image $u(\mathcal S_u)$ is $J_M$-holomorphic. Other than this codimension two subset, $X\setminus \mathcal S_u$ are regular points. Thus the restriction $u|_{X\setminus \mathcal S_u}$ is a local diffeomorphism. For any point in $M\setminus u(\mathcal S_u)$, the preimage contains degree $\deg u>0$ many points. Since near any point in $X\setminus (\mathcal S_u\cup u^{-1}(M\setminus u(\mathcal S_u)))$, $u$ is also a local diffeomorphism, the points in $u(X\setminus \mathcal S_u)\setminus (M\setminus u(\mathcal S_u))$ has at most $\deg u$ points in $X\setminus \mathcal S_u$.

Restricting to each irreducible component  of the singularity set, $\mathcal S_u$, the image of $u$ is either a point or a $J_M$-pseudoholomorphic $1$-subvariety. The union of the irreducible components contracted by $u$ constitutes of the subset $\mathcal S_u'\subset \mathcal S_u$.

The closure of each irreducible component $C$ of $\mathcal S_u\setminus \mathcal S_u'$, forgetting multiplicities, maps to $J_M$-holomorphic $1$-subvarieties. Choose a model Riemman surface $\Sigma$ for $C$. It gives a non-trivial pseudoholomorphic curve in $M$. In particular, the preimage of any point in the image is a set of finitely many points in $\Sigma$. 
\end{proof}
Hence, we have a more general version of Proposition \ref{maptorr} for somewhere immersed pseudoholomorphic maps.

\begin{prop}\label{4to4fiber}
Let $u: (X, J)\rightarrow (M, J_M)$ be a somewhere immersed pseudoholomorphic map between closed connected almost complex $4$-manifolds. Then other than finitely many points $x\in M$, where $u^{-1}(x)$ is the union of a $J$-holomorphic $1$-subvariety and finitely many points, the preimage of each point is a set of finitely many points.
\end{prop}
\begin{proof}
The singularity subset $\mathcal S_u$ is a (compact) $J$-holomorphic $1$-subvariety. Hence, only finitely many irreducible $J$-holomorphic subvarieties, which are some irreducible components of $\mathcal S_u$, will be contracted by $u$. The images are finitely many points in $M$. We denote the union of these irreducible $J$-holomorphic subvarieties by $\mathcal S_u'$. On $M\setminus \mathcal S_u'$, the preimage of any point $x\in u(M\setminus \mathcal S_u')$ under the map $u$ is a set of finitely many points by Lemma \ref{decX}.
\end{proof}

\subsection{Degree one pseudoholomorphic maps}
It is particularly interesting when applying the above results to a degree one pseudoholomorphic map $u: (X^4, J)\rightarrow (M^4, J_M)$.  Such a map has particularly nice structure which is essentially a birational morphism in pseudoholomorphic category.

First, we have a version of Zariski's main theorem for degree one pseudoholomorphic maps between almost complex $4$-manifolds.

\begin{prop}\label{Zarmain}
Let $u: (X, J) \rightarrow (M, J_M)$ be a degree one pseudoholomorphic map between closed connected almost complex $4$-manifolds. Then other than finitely many points $m\in M_1\subset M$, where $u^{-1}(m)$ is a connected $J$-holomorphic $1$-subvariety, $u|_{X\setminus u^{-1}(M_1)}$ is a diffeomorphism. 
\end{prop}
\begin{proof}
Each component of the $J$-holomorphic $1$-subvariety $\mathcal S_u$ is either contracted to a point or mapped onto a $J_M$-holomorphic $1$-subvariety, as in Lemma \ref{decX}. 

We can decompose $M$ into three parts. The first part, $M_1$, contains points whose preimages contain (non-trivial) $J$-holomorphic $1$-subvarieties. There are only finitely many such points. The second part, $M_2$, consists of points in $u(\mathcal S_u)$, but not in the first part. The closure of second part is a (compact) $J_M$-holomorphic $1$-subvariety. The third part, $M_3$, is the complement of the first two parts. It is an open subset of $M$ where $u$ is a diffeomorphism since $\deg u=1$.

We now show that, for any $m\in M_2$, $u^{-1}(m)$ is a single point. Suppose there are two points $x, y\in X$ such that $u(x)=u(y)=m$. Take an open neighborhood $\mathcal N_m\subset M$ of $m$, such that it does not contain any point in $M_1$. Then $\mathcal N_m\cap M_3$ is a complement of a relative compact set of codimension two. For any disjoint open neighborhoods $\mathcal N_x, \mathcal N_y\subset X$ of $x, y$, such that $u(\mathcal N_x), u(\mathcal N_y)\supset \mathcal N_m$, we know they overlap at a non-empty open subset $u^{-1}(\mathcal N_m\cap M_3)$. This contradicts the fact that $X$ is Hausdorff. 

Finally, by a similar argument, we can show that, for any $m\in M_1$, $u^{-1}(m)$ is a connected $J$-holomorphic $1$-subvariety. Suppose, other than a connected $J$-holomorphic $1$-subvariety $C$, there is an isolated point $x\in X$ or another connected $J$-holomorphic $1$-subvariety $C'$ disconnected from $C$ in $u^{-1}(m)$. We then similarly choose open neighborhood $\mathcal N_m\subset M$ of $m$ and disjoint neighborhoods $\mathcal N_x$ (or $\mathcal N_{C'}$) and $\mathcal N_C\subset X$ of $x$ (or $C'$) and $C$. Again, $u^{-1}(\mathcal N_m\cap M_3)$ would be a common open subset of both $\mathcal N_x$ (or $\mathcal N_{C'}$) and $\mathcal N_C$, contradicting  the fact that $X$ is Hausdorff. 

To summarize, for the finitely many points in $M_1$,  $u^{-1}(m)$ is a connected $J$-holomorphic $1$-subvariety. For any point in $M\setminus M_1=M_2\cup M_3$, the preimage under $u$ is a single point.

Since $\mathcal S_u=u^{-1}(M_1\cup M_2)$, now we want to show $M_2=\emptyset$, which would then imply $u|_{X\setminus u^{-1}(M_1)}$ is a diffeomorphism. As shown above, $u|_{u^{-1}(M_2)}$ is a degree one pseudoholomorhic map, which is non-singular except at possibly finitely many points. Denote these nonsingular points by $M_2'$. The differential $du$ is non-vanishing at the tangent direction of $TM_2'$. Then for any point $p\in M_2'$, and any complex direction $v\in T_pM$ which is not tangent to $M_2$, we choose an embedded $J_M$-holomorphic disk $D_v$ tangent to $v$ and such that $D_v\cap (M_1\cup M_2)=\{p\}$. The preimage $D_v':=u^{-1}(D_v)$ is also homeomorphic to a disk. Moreover, $u|_{D_v'}: D_v'\rightarrow D_v$ is a degree one map and holomorphic on $D_v'\setminus u^{-1}(p)$. Since the differential is bounded on $D_v'$, by the removable singularity theorem, $u|_{D_v'}: D_v'\rightarrow D_v$ is a holomorphic map. Since $D_v$ is embedded, the differential $du(v)\ne 0$. It implies $u$ is also diffeomorphic at $M_2'$. Hence $M_2'=\emptyset$ and there are only finitely many points in $M_2=M_2\setminus M_2'$. Since the closure of $M_2$ is a compact $J_M$-holomorphic $1$-subvariety, the only possibility is $M_2=\emptyset$.

We thus complete the proof of Zariski's main theorem for degree one pseudoholomorphic maps between closed almost complex $4$-manifolds.
\end{proof}

We will continue the notation $M_i\subset M, i=1, 2, 3$, defined in the above proof in the following. By the above discussion, it is clear that any degree one pseudoholomorphic map $u: (X, J) \rightarrow (M, J_M)$ is a composition of $\#\{M_1\}$ many pseudoholomorphic maps such that each of them has only one point whose preimage is not a point. Let us study the preimage of $M_1$.

\begin{prop}\label{rcpre}
Let $u: (X, J) \rightarrow (M, J_M)$ be a degree one pseudoholomorphic map between closed almost complex $4$-manifolds. Then the fundamental group of the $J$-holomorphic $1$-subvariety $u^{-1}(M_1)$ is trivial. Equivalently, each irreducible component of the $J$-holomorphic $1$-subvariety $u^{-1}(M_1)$ is a smooth rational curve and there are no cycles enclosed by different irreducible components.
\end{prop}
\begin{proof}
For any point $m\in M_1$, denote the $J$-holomorphic $1$-subvariety $u^{-1}(m)$ by $C_m$.  The configuration of $C_m$ might have triple intersections or tangent points of different irreducible components, or singular points of some irreducible components. We first topologically blow up the configuration of $C_m$ at these points. We denote the new surface configuration by $C_m'$. For any tubular neighborhoods $A_1$ of $C_m$ and $A_2$ of $C_m'$, the following pairs are homeomorphic $A_1\setminus C_m \cong A_2\setminus C_m'$ and $\partial A_1\cong \partial A_2$. The statement of our proposition follows if we can show $H^1(C_m')=0$ for the new surface configuration. 

 For any open ball $B_{\delta}$ around $m$ such that $B_{\delta}\cap M_1=\{m\}$, let $K_{\delta}=u^{-1}(B_{\delta})$.  Denote by $K_{\delta}'$ the image of $K_{\delta}$ after topological blowup.  We have $\pi_1(K_{\delta}'\setminus C_m')=\pi_1(K_{\delta}\setminus C_m)=\pi_1(B_{\delta}\setminus \{m\})=1$. 
There exists a tubular neighborhood $A'$ of $C_m'$, such that $A'\subset K_{\delta}'\subset A$ where $A$ is obtained from $A'$ by multiplying the normal distance by a fixed positive number $r>1$. Thus any path in $A\setminus C_m'$ is homotopic to a path in $A'\setminus C_m'$ which is nullhomotopic in $A\setminus C_m'$ because $\pi_1(K_{\delta}'\setminus C_m')=1$. This implies $\pi_1(\partial A)=1$ and it is independent of the choice of the neighborhood. Hence we say the ``boundary fundamental group" of $C_m'$ is trivial.

The tubular neighborhood $A$  deformation retracts to $C_m'$. Thus $H_i(A)=H_i(C_m')$. We have the exact sequence $$H^1(A, \partial A)\rightarrow H^1(A)\rightarrow H^1(\partial A)=0.$$
By Poincar\'e duality $H^1(A, \partial A)=H_3(A)=0$. Hence $H^1(C_m')=H^1(A)=0$ 
and all the conclusions follow.
\end{proof}

The topological  blowdown is the reverse action of the topological blowup. It contracts a smooth sphere of self-intersection $-1$. The topological blowup and blowdown for surface configurations can actually be performed holomorphically, at least locally, for  pseudoholomorphic $1$-subvarieties (see Theorem $3$ in \cite{Sik} and the proof of Theorem \ref{phblowdown}). Later, we say a surface configuration (resp. a pseudoholomorphic $1$-subvariety) is {\it equivalent} to another if they are related by topological (resp. locally holomorphic) blowups and blowdowns.

To describe a finer structure of a pseudoholomorphic exceptional subset, we introduce the following definition, see \cite{BZ, ES}. 

\begin{definition}\label{exc1st}
An exceptional curve of the first kind in a complex surface $X$ is a divisor $E$ for which there is a birational map $\pi: X\rightarrow M$ to a smooth complex surface $M$ and a point $m\in M$ such that $\pi^{-1}(m)=E$. 

Let $X$ be a $4$-manifold, and $\Theta$ the image of a continuous map from a nodal Riemann surface of genus zero. We say that $\Theta$ is an exceptional curve of the first kind if there exists a neighborhood $(N, J)$ of $\Theta$ where $J$ is an integrable complex structure, an open neighborhood $N'$ of $0\in \mathbb C^2$, and a holomorphic birational map $\pi: N\rightarrow N'$, such that $\Theta$ is a $J$-holomorphic $1$-subvariety and $\pi^{-1}(0)=\Theta$.
\end{definition}

The exceptional curves of the first kind are treated systematically in \cite{BZ}. Any such curve could be obtained by a sequence of blowups at points. It implies that any two irreducible components intersect at most once transversely and have no triple intersections. There is at least one $-1$-sphere amongst the irreducible components. The following should be well known.

\begin{lemma}\label{excbd}
Blowing down a $-1$-sphere from an exceptional curve of the first kind gives another exceptional curve of the first kind.
\end{lemma}
\begin{proof}
Essentially, we only need to show that we can continue the blowdown process until no curves are left, whatever the order of blowdowns we choose. We call the starting exceptional curve of the first kind $\Theta=\{(C_i, m_i)\}_{i=1}^n$.

Since an exceptional curve of the first kind is obtained from consecutive blowups, the cohomology class of the corresponding pseudoholomorphic subvariety is a $-1$-rational curve class $E_1$. In fact, this process is equivalent to the following choices of homology classes: there are $n$ second cohomology classes $E_1, \cdots, E_n$ which can be represented by embedded symplectic spheres with $E_i^2=-1$, $i=1, \cdots, n$, and $E_i\cdot E_j=0$ for $i\ne j$, such that $[C_i]=E_i-\sum_{j> i} m_{ij}E_j$. Here $(m_{ij})_{i, j=1}^n$ is a strict upper triangular square matrix whose entries are $0$ or $1$. The connectedness of $\Theta$ is equivalent to saying that there does not exist $j$ such that $m_{ij}=m_{ji}=0, \forall i=1, \cdots, n$. Moreover, we have $\sum m_i[C_i]=E_1$. The inverse of the original blowup process to get $\Theta$ is to blow down $C_n, \cdots, C_1$ consecutively. 

We can also start blowing down from some $C_l$ where $m_{lj}=0$ for $\forall j>l$. After that, we can relabel $C_i, E_i$ from $1$ to $n-1$ without changing the original order. The new matrix $(m'_{ij})_{i, j=1}^{n-1}$ is obtained from $(m_{ij})_{i, j=1}^n$ by deleting $l^{th}$-row and column. Hence, the new subvariety obtained from blowing down $C_l$ is still an exceptional curve of the first kind.
 \end{proof}

We end the digression. 

The following theorem shows that a degree $1$ map is a suitable alternative of a blowdown morphism in the pseudoholomorphic setting. It gives an affirmative answer to Question 5.4 of \cite{Zadv}.

\begin{theorem}\label{phblowdown}
Let $u: (X, J) \rightarrow (M, J_M)$ be a degree one pseudoholomorphic map between closed almost complex $4$-manifolds such that $J$ is compatible with a symplectic structure $\omega$ on $X$. Then other than finitely many points $M_1\subset M$, $u|_{X\setminus u^{-1}(M_1)}$ is a diffeomorphism. At each point of $M_1$, the preimage supports an exceptional curve of the first kind. 
\end{theorem}
\begin{proof}
Other than the last conclusion, the statements follow from Propositions \ref{Zarmain} and \ref{rcpre}. 

For any point $m\in M_1$, we still denote the $J$-holomorphic $1$-subvariety $u^{-1}(m)$ by $C_m$. We denote the intersection matrix for $C_m$ by $Q_{C_m}$. It is a symmetric square matrix whose size is the number of irreducible components of $C_m$. We want to show that this matrix $Q_{C_m}$ is negative definite, as a generalization of Grauert's criterion for exceptional sets of analytic maps. 

We use a gluing result of McCarthy-Wolfson \cite{McW}. Let $Y$ be $(2n-1)$-dimensional submanifold of a $2n$-dimensional symplectic manifold $(X, \omega_X)$. Suppose $Y$ admits a fixed point free $S^1$ action. The manifold $Y$ is called $\omega_X$-compatible if the orbits of the action lie in the null direction of $\omega_X|_Y$. If $Y$ is a separating hypersurface, let $X^-$ be the piece for which $Y$ is the $\omega_X$-convex boundary and $X^+$ be the other piece. Then we have the following gluing result. We only state it for $\dim Y=3$, since it is what we need in the proof.

\begin{theorem}[McCarthy and Wolfson]\label{McWglue}
Let $Y$ be a Seifert $3$-manifold, {\it i.e.} a compact manifold with a fixed point free $S^1$ action. Let $(X_i, \omega_i)$, $i=1, 2$, be symplectic $4$-manifolds,  and suppose that there are $\omega_i$-compatible embeddings $j_i: Y\rightarrow (X_i, \omega_i)$ such that $j_i(Y)$ is a separating hypersurface in $X_i$. Then there is a symplectic structure $\omega$ on $\tilde X=X_1^-\cup_Y X_2^+$. Moreover, there are neighborhoods $N_i(Y)$ of $Y$ in $X_i$ such that $\omega=\omega_2$ on $X_2^+\setminus N_2(Y)$ and $\omega=c\omega_1$ on $X_1^-\setminus N_1(Y)$ for some positive constant $c$.
\end{theorem}

We choose an open ball neighborhood $\mathcal N_m$ of $m\in M_1$ such that the boundary is $J_M$-convex and $\mathcal N_m\cap M_1=\{m\}$. We can choose $\mathcal N_m$ such that it is contained in a neighborhood of $m$ such that there exists a symplectic form compatible with $J_M$. The induced contact structure on $\partial \mathcal N_m$ is the unique tight contact structure on $S^3$. This is our $Y$ whose fixed point free $S^1$-action is induced by Reeb orbits. Hence, $\mathcal N_m$ could be capped (by a concave neighborhood of $+1$-sphere) to a symplectic $\mathbb CP^2$. This is our $X_2$. 

We take $X_1$ to be our $X$ with the $J$-compatible symplectic form $\omega_X$. The preimage $u^{-1}(\partial\mathcal N_m)$ is diffeomorphic to $S^3$. Moreover, since $u$ is pseudoholomorphic and $u|_{u^{-1}(\partial\mathcal N_m)}$ is a diffeomorphism, the $J$-lines provide a contact structure on $u^{-1}(\partial\mathcal N_m)$ which is contactomorphic to the one on $\partial \mathcal N_m$ induced by $J_M$-lines. Hence, we can apply Theorem \ref{McWglue} to obtain a symplectic manifold $\tilde X=X_1^-\cup_Y X_2^+$. 

Since $X_2^+$ contains a symplectic sphere $S$ of self-intersection $1$, the symplectic $4$-manifold $\tilde X$ is diffeomorphic to $\mathbb CP^2\# k\overline{\mathbb CP^2}$, {\it i.e.} a rational $4$-manifold \cite{McD90}. In particular, $b^+(\tilde X)=1$. Since $C_m$ is disjoint from $S$, we have $Q_{C_m\cup S}=Q_{\C_m}\oplus (1)$ as a sub matrix of the intersection matrix of $\tilde X$. It implies the matrix $Q_{C_m}$ corresponding to $C_m$ is negative definite.

Hence, we can apply Proposition 4.4 of \cite{LM} (its proof eventually follows from \cite{Sha}). We first topologically blow up our configuration $C_m$ such that every intersection of the new configuration $C_m'$ is transverse and there is no triple intersection. We can actually realize this step by complex blowups. By Theorem $3$ of \cite{Sik}, there exists another (tamed) almost complex structure $J'$ on $X$ for which $C_m$ is $J'$-holomorphic and $J'$ is integrable on a neighborhood of $C_m$. Then we can apply the above topological blowups in a complex way for this almost complex structure $J'$.

After blowups, $Q_{C_m'}$ is still negative definite.  For a  surface configuration $C$ with each intersection transverse and having no triple intersections, one can associate a weighted finite graph $\Gamma_C$ whose vertices represent the surfaces and each edge joining two vertices represents an intersection between the two surfaces corresponding to the two vertices. Moreover, each vertex is weighted by its genus and its self-intersection number. 

Applying Proposition 4.4 of \cite{LM} to the graph $\Gamma_{C_m'}$, it implies $\Gamma_{C_m'}$ is equivalent to a graph of type (N) as listed in \cite{LM}. The type (N) graphs have three kinds. The graph (N1) is the empty graph. Its boundary fundamental group is trivial. The type (N2) graphs are linear graphs. The boundary fundamental groups are non-trivial cyclic groups. The type (N3) graphs are star shapes with one branching point and three branches. The boundary fundamental groups are non-cyclic finite groups.  Since boundary fundamental groups of types (N2) and (N3) are non trivial, we know the underlying graph $\Gamma_{C_m'}$ is equivalent to the empty graph, {\it i.e.} type (N1). In fact, it is equivalent to saying that $C_m'$ is an exceptional curve of the first kind with the multiplicities induced by the process of blowups. 

The $J'$-holomorphic (and also $J$-holomorphic) subvariety $C_m$ is obtained from $C_m'$ by a sequence of complex blowdowns. By Lemma \ref{excbd}, $C_m$ is also an exceptional curve of the first kind. In particular, it implies every intersection of $C_m$ is transverse and there is no triple intersection.

This completes our proof.
\end{proof}

\begin{cor}\label{diffbd}
Let $u: (X, J) \rightarrow (M, J_M)$ be a degree one pseudoholomorphic map between closed almost complex $4$-manifolds such that $J$ is compatible with a symplectic structure $\omega$ on $X$. Then $X=M\#k\overline{\mathbb CP^2}$ diffeomorphically, where $k$ is the number of irreducible components of the $J$-holomorphic $1$-subvariety $u^{-1}(M_1)$.
\end{cor}
\begin{proof}
It clearly follows from Theorem \ref{phblowdown}. Moreover, from the proof of Lemma \ref{excbd}, for the exceptional curve of the first kind $C_m=u^{-1}(m)$, $\forall m\in M_1$, with $n(m)$ irreducible components, we have $n(m)$ cohomology classes $E_1^m, \cdots, E^m_{n(m)}$ in $H^2(X, \mathbb Z)$ which can be represented by embedded symplectic spheres with $(E_i^m)^2=-1$, $i=1, \cdots, n(m)$, and $E_i^m\cdot E_j^m=0$ for $i\ne j$, such that the cohomology class of each irreducible component is $E_i^m-\sum_{j> i} m_{ij}E_j^m$.

Hence, each $\overline{CP^2}$ corresponds to such an $E_j^m$, and $$H^2(X, \mathbb Z)=u^*(H^2(M, \mathbb Z))\bigoplus_{m\in M_1}(\bigoplus_{i=1}^{n(m)} \mathbb Z E^m_{i}),$$ with $k=\sum_{m\in M_1} n(m)$.
\end{proof}

\subsection{Symplectic birational geometry in dimension $6$}\label{symbir}
We could also use our results, mainly Corollary \ref{eICdim4}, to study symplectic birational geometry in dimension $6$. We assume our ambient manifold $M$ is closed.

\begin{prop}\label{disss}
Let $Z_1, Z_2$ be two embedded symplectic $4$-manifolds in a $6$-dimensional symplectic manifold $(M, \omega)$. If $PD[Z_1]\cup PD[Z_2]\cup [\omega]\le 0$, then $Z_1$ and $Z_2$ are $J$-holomorphic simultaneously if and only if $Z_1$ and $Z_2$ are disjoint (in which case we have $[Z_1]\cdot [Z_2]=0$).
\end{prop}
\begin{proof}
First when $Z_1$ and $Z_2$ are disjoint, $Z_1\cup Z_2$ is an embedded symplectic submanifold in $(M, \omega)$. Hence we can realize it as a $J$-holomorphic submanifold for some $J$ tamed by $\omega$. And we have $[Z_1]\cdot [Z_2]=0$.

On the other hand, if $Z_1$ and $Z_2$ are $J$-holomorphic simultaneously. Then by Corollary \ref{eICdim4}, we know the intersection $Z_1\cap Z_2$ is a $J$-holomorphic $1$-subvariety in class $[Z_1]\cdot [Z_2]$. Since $J$ is tamed, we must have $PD[Z_1]\cup PD[Z_2]\cup [\omega]\ge 0$. The equality holds if and only if $[Z_1]\cdot [Z_2]=0$, which implies $Z_1$ and $Z_2$ are disjoint. 
\end{proof}

Conversely, two $J$-holomorphic submanifolds $Z_1^4, Z_2^4$ in $(M^6, J)$ with $J|_{Z_1}$ tamed and $[Z_1]\cdot [Z_2]=0$ cannot intersect. Otherwise, the intersection is a $J$-holomorphic subvariety in $Z_1$, by Corollary \ref{eICdim4}, whose homology class is non-trivial since $J|_{Z_1}$ is tamed. 

If a smooth $J$-holomorphic curve in a $4$-dimensional almost complex manifold has negative self-intersection, then there is no other $J$-holomorphic curve in the same homology class. This follows from positivity of intersections. What follows is a generalization of this fact in dimension $6$.
\begin{prop}\label{-int}
If $(Z^4, \omega)$ is a $4$-dimensional symplectic manifold embedded in $(M^6, J)$ as a $J$-holomorphic submanifold such that $J|_Z$ is tamed by $\omega$, and  $c_1(N_M(Z))\cdot [\omega]< 0$, then there is no other almost complex submanifold in $(M, J)$ which is homologous to $Z$.
\end{prop}
\begin{proof}
If there is another such one $Z'$, 
then the homology class of their intersection $Z\cap Z'$ in $H_2(Z, \mathbb Z)$ could be calculated by looking at the intersection of a submanifold $Z''\subset M$ such that $Z\pitchfork Z''$ and $[Z']=[Z'']$. Since all such $Z''$ intersect $Z$ in a manifold with the same homology class in $H_2(Z, \mathbb Z)$, we could choose $Z''$ to be a smooth perturbation of $Z$ in a small neighborhood, which is identified with its normal bundle $N_M(Z)$, such that $Z\pitchfork Z''$. The intersection $V=Z\cap Z''$ is a submanifold of $Z$. 
Hence the homology class of $V$ in $H_2(Z, \mathbb Z)$ is Poincar\'e dual to $c_1(N_M(Z))$. Combining with Corollary \ref{eICdim4}, this implies the intersection $Z\cap Z'$ is a $J$-holomorphic subvariety of $Z$ whose homology class is the Poincar\'e dual of $c_1(N_M(Z))$. Since $J$ is tamed by $\omega$, we have $c_1(N_M(Z))\cdot [\omega]\ge 0$ which contradicts  our assumption.  
\end{proof}

In a non minimal symplectic $4$-manifold $(M, \omega)$, {\it i.e.} when $M$ contains a smooth sphere of self-intersection $-1$, there are infinitely many embedded symplectic spheres in any exceptional class $E$. However, there is at most one smooth $J$-holomorphic curve in class $E$ if we fix an almost complex structure $J$ tamed by $\omega$. This is also true in dimension $6$.  
\begin{definition}\label{budiv}
An almost complex submanifold $D\subset (M^6, J)$ is called a smooth blow-up divisor if it is either $\mathbb CP^2$ or a $\mathbb CP^1$ bundle over a Riemann surface $\Sigma_g$, whose normal line bundle $N_M(D)$ is $\mathcal O(-1)$ along $\mathbb CP^1\subset \mathbb CP^2$ or the fibers of $\mathbb CP^1$ bundles over $\Sigma_g$. It is further called an almost complex blow-up divisor, if each fiber is a smooth $J$-holomorphic curve. 
\end{definition}
Since $D$ is diffeomorphic to $\mathbb CP^2$ or a $\mathbb CP^1$ bundle over $\Sigma_g$, we assume the homology class of $\mathbb CP^1\subset \mathbb CP^2$ or the fibers of the latter case to be $F$.

 Any complex or symplectic divisor arising from a complex or symplectic blow-up is a smooth blow-up divisor.  
We remark that even if there is a symplectic divisor $D$ in a symplectic manifold which is an almost complex blow-up divisor for some tamed $J$, it might not have a symplectic structure on the blown down manifold which is compatible with the almost complex structure. Think about a Moishezon manifold. 

\begin{prop}\label{uniE}
Let an almost complex submanifold $D\subset (M^6, J)$ be a smooth blow-up divisor. Suppose there is an irreducible curve inside $D$ in the class $F$. Then there is no other almost complex submanifold of $M$ in class $[D]$. 
\end{prop}
\begin{proof}
As in the proof of Proposition \ref{-int}, if there is another such divisor $D'$, then the homology class of the intersection $D\cap D'$ is Poincar\'e dual to $c_1(N_M(D))\in H^2(D, \mathbb Z)$. However, since $D$ is a smooth blow-up divisor, $c_1(N_M(D))\cdot F=-1$. There is no $J|_D$-holomorphic subvariety in such a class $c_1(N_M(D))$ since $F$ pairs non-negatively with any $J|_D$-holomorphic subvariety, which is because $F$ is represented by a smooth $J|_D$-holomorphic curve of non-negative self-intersection in $D$.  
This contradiction implies there is no other almost complex submanifold of $M$ in class $[D]$. 
\end{proof}

The argument for Proposition \ref{uniE} works in more general setting, {\it e.g.} we still have uniqueness when $N_M(D)$ is $\mathcal O(-k)$, $k>0$,  along $\mathbb CP^1\subset \mathbb CP^2$ or the fibers of $\mathbb CP^1$ bundles over $\Sigma_g$. Under these assumptions,  it is known that there is a contraction in the complex analytic setting, {\it i.e.} a proper surjective holomorphic mapping $f: M\rightarrow N$ onto a complex analytic variety $N$, such that $f|_D: D\rightarrow B$ is the fibration of the ruled surface where $B$ is either a point or $\Sigma_g$, and $f:M\setminus D\rightarrow N\setminus B$ is an isomorphism.

\end{document}